\documentclass[10pt,a4paper]{article}
\usepackage[utf8]{inputenc}
\usepackage[T1]{fontenc}
\usepackage{geometry}
\usepackage{amsthm}
\usepackage[francais,english]{babel}
\usepackage{amssymb}
\usepackage{lmodern}
\usepackage{calligra,amsmath,amsfonts}
\usepackage{mathrsfs} 
\usepackage{dsfont}
\usepackage{stmaryrd}

\usepackage{graphicx}
\usepackage{fullpage}
\usepackage[nottoc,notlot,notlof]{tocbibind}
\usepackage[colorlinks=true,urlcolor=black,linkcolor=black,citecolor=black]{hyperref}
\numberwithin{equation}{section}
\usepackage{tikz-cd}
\usetikzlibrary{matrix}
\usepackage{mathtools}
\usepackage{verbatimbox}
\usepackage{diagbox}

\begin{document}
	
	\title{On the Morse Index of Critical Points in the Viscosity Method}
	\author{Alexis Michelat\footnote{Department of Mathematics, ETH Zentrum, CH-8093 Z\"{u}rich, Switzerland.}\setcounter{footnote}{0}}
	\date{\today}
	
	\maketitle
	
	\vspace{1.5em}
	
	\begin{abstract}
		We show that in viscous approximations of functionals defined on Finsler manifolds, it is possible to construct suitable sequences of critical points of these approximations satisfying the expected Morse index bounds as in Lazer-Solimini's theory, together with the entropy condition of Michael Struwe.
	\end{abstract}

	\tableofcontents
	\vspace{0.5cm}
	\begin{center}
	\end{center}
	\theoremstyle{plain}
	\newtheorem*{theorem*}{Theorem}
	\newtheorem{theorem}{Theorem}[section]
	\newtheorem{lemme}[theorem]{Lemma}
	\newtheorem{propdef}[theorem]{Proposition-Definition}
	\newtheorem{prop}[theorem]{Proposition}
	\newtheorem{cor}[theorem]{Corollary}
	\theoremstyle{definition}
	\newtheorem*{definition}{Definition}
	\newtheorem*{definitions}{Definitions}
	\newtheorem{defi}[theorem]{Definition}
	\newtheorem{rem}[theorem]{Remark}
	\newtheorem*{rem*}{Remark}
	\newtheorem{remimp}[theorem]{Important remark}
	\newtheorem{rems}[theorem]{Remarks}
	\newtheorem*{rems2}{Remarks}
	\newtheorem{exemple}[theorem]{Example}
	\newcommand{\N}{\ensuremath{\mathbb{N}}}
	\parskip 1ex
	\newcommand{\vc}[3]{\overset{#2}{\underset{#3}{#1}}}
	\newcommand{\conv}[1]{\ensuremath{\underset{#1}{\longrightarrow}}}
	\newcommand{\A}{\ensuremath{\mathscr{A}}}
	\newcommand{\D}{\ensuremath{\nabla}}
	\renewcommand{\N}{\ensuremath{\mathbb{N}}}
	\newcommand{\Z}{\ensuremath{\mathbb{Z}}}
	\newcommand{\I}{\ensuremath{\mathbb{I}}}
	\newcommand{\K}{\ensuremath{\mathbb{K}}}
	\newcommand{\R}{\ensuremath{\mathbb{R}}}
	\newcommand{\W}{\ensuremath{\mathscr{W}}}
	\newcommand{\Q}{\ensuremath{\mathscr{Q}}}
	\newcommand{\C}{\ensuremath{\mathbb{C}}}
	\newcommand{\z}{\ensuremath{\bar{z}}}
	\renewcommand{\tilde}{\ensuremath{\widetilde}}
	\newcommand{\p}[1]{\ensuremath{\partial_{#1}}}
	\newcommand{\Res}{\ensuremath{\mathrm{Res}}}
	\newcommand{\lp}[2]{\ensuremath{\mathrm{L}^{#1}(#2)}}
	\renewcommand{\wp}[3]{\ensuremath{\left\Vert #1\right\Vert_{\mathrm{W}^{#2}(#3)}}}
	\newcommand{\hp}[3]{\ensuremath{\left\Vert #1\right\Vert_{\mathrm{H}^{#2}(#3)}}}
	\newcommand{\np}[3]{\ensuremath{\left\Vert #1\right\Vert_{\mathrm{L}^{#2}(#3)}}}
	\newcommand{\h}{\ensuremath{\vec{h}}}
	\renewcommand{\Re}{\ensuremath{\mathrm{Re}\,}}
	\renewcommand{\Im}{\ensuremath{\mathrm{Im}\,}}
	\newcommand{\diam}{\ensuremath{\mathrm{diam}\,}}
	\newcommand{\leb}{\ensuremath{\mathscr{L}}}
	\newcommand{\supp}{\ensuremath{\mathrm{supp}\,}}
	\renewcommand{\phi}{\ensuremath{\vec{\Phi}}}
	\newcommand{\Perp}{\ensuremath{\perp}}
	\renewcommand{\H}{\ensuremath{\vec{H}}}
	\newcommand{\norm}[1]{\ensuremath{\Vert #1\Vert}}
	\newcommand{\e}{\ensuremath{\vec{e}}}
	\newcommand{\f}{\ensuremath{\vec{f}}}
	\renewcommand{\epsilon}{\ensuremath{\varepsilon}}
	\renewcommand{\bar}{\ensuremath{\overline}}
	\newcommand{\s}[2]{\ensuremath{\langle #1,#2\rangle}}
	\newcommand{\bs}[2]{\ensuremath{\left\langle #1,#2\right\rangle}}
	\newcommand{\n}{\ensuremath{\vec{n}}}
	\newcommand{\ens}[1]{\ensuremath{\left\{ #1\right\}}}
	\newcommand{\w}{\ensuremath{\vec{w}}}
	\newcommand{\vg}{\ensuremath{\mathrm{vol}_g}}
	\renewcommand{\d}[1]{\ensuremath{\partial_{x_{#1}}}}
	\newcommand{\dg}{\ensuremath{\mathrm{div}_{g}}}
	\renewcommand{\Res}{\ensuremath{\mathrm{Res}}}
	\newcommand{\totimes}{\ensuremath{\,\dot{\otimes}\,}}
	\newcommand{\un}[2]{\ensuremath{\bigcup\limits_{#1}^{#2}}}
	\newcommand{\res}{\mathbin{\vrule height 1.6ex depth 0pt width
			0.13ex\vrule height 0.13ex depth 0pt width 1.3ex}}
	\newcommand{\ala}[5]{\ensuremath{e^{-6\lambda}\left(e^{2\lambda_{#1}}\alpha_{#2}^{#3}-\mu\alpha_{#2}^{#1}\right)\left\langle \nabla_{\vec{e}_{#4}}\vec{w},\vec{\mathbb{I}}_{#5}\right\rangle}}
	\setlength\boxtopsep{1pt}
	\setlength\boxbottomsep{1pt}
	\allowdisplaybreaks
	\newcommand*\mcup{\mathbin{\mathpalette\mcapinn\relax}}
	\newcommand*\mcapinn[2]{\vcenter{\hbox{$\mathsurround=0pt
				\ifx\displaystyle#1\textstyle\else#1\fi\bigcup$}}}
	\def\Xint#1{\mathchoice
		{\XXint\displaystyle\textstyle{#1}}%
		{\XXint\textstyle\scriptstyle{#1}}%
		{\XXint\scriptstyle\scriptscriptstyle{#1}}%
		{\XXint\scriptscriptstyle\scriptscriptstyle{#1}}%
		\!\int}
	\def\XXint#1#2#3{{\setbox0=\hbox{$#1{#2#3}{\int}$ }
			\vcenter{\hbox{$#2#3$ }}\kern-.58\wd0}}
	\def\ddashint{\Xint=}
	\newcommand{\dashint}[1]{\ensuremath{{\Xint-}_{\mkern-10mu #1}}}
	\newcommand{\xrightarrowdbl}[2][]{%
		\xrightarrow[#1]{#2}\mathrel{\mkern-14mu}\rightarrow
	}

\section{Introduction}

In this paper, we want to show that one can construct critical points of the right index depending on the dimension of the admissible min-max family in the framework of the viscosity method. Namely, we fix a $C^2$ Finsler manifold $X$ and we consider a $C^2$ function $F:X\rightarrow \R$, for which one aims at constructing (unstable) critical points. We further fix some $d$-dimensional compact manifold $M^d$  with boundary $\partial M^d= B^{d-1}\neq \varnothing$, and a continuous map $h: B^{d-1}\rightarrow X$, and we call the subset $\mathscr{A}\subset \mathscr{P}(X)$ a $d$-dimensional admissible family (relative to $(M^d,h)$) if
\begin{align*}
	\mathscr{A}=\mathscr{P}(X)\cap\ens{A: \;\text{there exists a continuous map}\;\, f:M^d\rightarrow X\;\, \text{such that}\;\, A=f(M^d)\;\, \text{and}\;\, f_{B^{d-1}}=h}.
\end{align*}
We shall generalise this example later and introduce additional min-max families in Section \ref{min-max}.
In particular, notice that $\mathscr{A}$ is stable under homeomorphisms isotopic to the identity preserving the boundary $h(B^{d-1})\subset X$.
Then the min-max level associated to $F$ and $\mathscr{A}$, denoted here by $\beta(F,\mathscr{A})$ or ($\beta(\mathscr{A})$ when there is no ambiguity in the choice of $F$)
\begin{align*}
	\beta(F,\mathscr{A})=\inf_{A\in \mathscr{A}}\sup F(A)<\infty.
\end{align*}
Assuming that the min-max is non-trivial in the following sense
\begin{align*}
	\beta(\mathscr{A})=\inf_{A\in \mathscr{A}}\sup F(A)>\sup F(h(B))=\widehat{\beta}(\mathscr{A}),
\end{align*}
this is a very classical theorem of Palais (\cite{palais}) 
that there exists a critical point $x\in K(F)$ of $F$ such that $F(x)=\beta(\mathscr{A})$, provided $F$ satisfies the celebrated Palais-Smale (PS) condition. 

Now, we assume furthermore that $X$ is a Finsler-Hilbert manifold and that the linear map $\D^2 F(x):T_xX\rightarrow T_xX$ is a Fredholm operator at every critical point $x\in K(F)$. We also define the index $\mathrm{Ind}_F(x)\in \N$ (resp. the nullity $\mathrm{Null}_F(x)$) of a critical point $x\in K(F)$ of $F$ as the number (with multiplicity) of negative eigenvalues (resp. as the multiplicity of the $0$-eigenvalue) of the Fredholm operator $\D^2 F(x):T_xX\rightarrow T_xX$.

 In this setting, it was subsequently proved by Lazer and Solimini (\cite{lazer}) that it is possible to find a critical point $x^{\ast}\in K(F)$ (\textit{a priori} different from $x$) such that we get the following index bound
\begin{align}\label{ib}
	\mathrm{Ind}_{F}(x^{\ast})\leq d.
\end{align}
For different types of min-max family, it is also possible to obtain a one-sided bound 
\begin{align*}
	\mathrm{Ind}_{F}(x^{\ast})+\mathrm{Null}_{F}(x^{\ast})\geq d
\end{align*}
or a two-sided estimate
\begin{align*}
	\mathrm{Ind}_{F}(x^{\ast})\leq d\leq \mathrm{Ind}_{F}(x^{\ast})+\mathrm{Null}_{F}(x^{\ast}).
\end{align*}
In particular, if $F$ is non-degenerate at $x$, we obtain a critical point for the third kind of families of index exactly equal to $d$ (to be defined in Section \ref{min-max}). For min-max families defined with respect to homology classes, the two-sided estimate was first obtained by Claude Viterbo (\cite{viterbo}).

Now, in the framework of the viscosity method (see \cite{geodesics} for a general introduction on the subject), the function $F$ does not satisfy the Palais-Smale condition (take for example minimal or Willmore surfaces), and one wishes to construct critical points of $F$ by approaching $F$ by a more coercive function for which we can apply the previous standard methods. We let $G:X\rightarrow \R_+$ be a $C^2$ function and we define for all $\sigma>0$ the $C^2$ function $F_{\sigma}=F+\sigma^2G$, and we assume that for all $\sigma>0$, the function $F_{\sigma}:X\rightarrow \R$ verifies the Palais-Smale condition. Furthermore, we denote for all $\sigma\geq 0$ (so that $\beta(0)=\beta(\mathscr{A})$)
\begin{align*}
	\beta(\sigma)=\beta(F_{\sigma},\mathscr{A})\geq \beta(0)=\beta(\mathscr{A}).
\end{align*}
In particular, the previous theory applies and we can find for all $\sigma>0$ a critical point $x_{\sigma}$ of $F_{\sigma}$ of the right index. Then this is a case-by-case analysis to show that the bounds carry one as $\sigma\rightarrow 0$ (see \cite{hierarchies} for minimal surfaces and \cite{morse2} for Willmore surfaces). However, the first problem which might occur (and actually the only one) is to loose energy in the approximation part, \text{i.e.} to have for some sequence $\ens{\sigma_k}_{k\in \N}\subset (0,\infty)$ converging towards $0$ and some sequence $\ens{x_k}_{k\in \N}\subset X$ of critical points associated to $\ens{F_{\sigma_k}}_{k\in \N}$ (\textit{i.e.} such that $x_k\in K(F_{\sigma_k},\beta(\sigma_k))$ for all $k\in \N$)
\begin{align*}
	F_{\sigma_k}(x_k)=\beta(\sigma_k)\conv{k\rightarrow 0}\beta(0)=\beta(\mathscr{A}),\quad \text{and}\;\, \limsup_{k\rightarrow \infty}F(x_{\sigma_k})<\beta(\mathscr{A}).
\end{align*}
There are some explicit examples of such failure (see \textit{e.g.} \cite{geodesics} for examples for geodesics and minimal surfaces), but Michael Struwe found that this  was possible to overcome this difficulty through what is called \emph{Struwe's monotonicity trick} (see \cite{struwe}, \cite{struwe2}). In our setting, the corresponding theorem is the following (see \cite{geodesics} or \cite{rivcours} for a proof).
\begin{theorem*}[$\ast$]
	Let $(X,\norm{\,\cdot\,})$ be a complete $C^1$ Finsler manifold. Let $F_{\sigma}:X\rightarrow \R$ be a family of $C^1$ functions for all $\sigma\in [0,1]$ such that for all $x\in X$, 
	\begin{align*}
		\sigma\mapsto F_{\sigma}(x)
	\end{align*}
	 is $C^1$ and increasing. Assume furthermore that there exists $C\in L^{\infty}_{\mathrm{loc}}((0,1))$, $\delta\in L^{\infty}_{\mathrm{loc}}(\R_+)$ going to $0$ at $0$, and $f\in L^{\infty}_{\mathrm{loc}}(\R)$ such that for all $0<\sigma,\tau<1$ and for all $x\in X$,
	\begin{align}\label{pseudo}
		\norm{D F_{\tau}(x)-DF_{\sigma}(x)}_x\leq C(\sigma)\delta(|\sigma-\tau|)f(F_{\sigma}(x)).
	\end{align}
	Finally, assume that for $\sigma>0$ the function $F_{\sigma}$ satisfies the Palais-Smale condition. Let $\mathscr{A}$ be an admissible family of min-max of $X$ and denote 
	\[
	\beta(\sigma)=\beta(F_{\sigma},\mathscr{A})=\inf_{A\in \mathscr{A}}\sup F_{\sigma}(A).
	\]
	Then there exists a sequence $\ens{\sigma_k}_{k\in \N}\subset (0,\infty)$ and $\ens{x_k}_{k\in \N}\subset X$ such that 
	\begin{align*}
		F_{\sigma_k}(x_k)=\beta(\sigma_k),\quad DF_{\sigma_k}(x_k)=0.
	\end{align*}
	Furthermore, for all $k\in \N$, the critical point $x_k$ satisfies the following \emph{entropy condition}
	\begin{align}\label{boltzmann}
		\p{\sigma_k}F_{\sigma_k}(x_k)\leq \frac{1}{\sigma_k\log\left(\frac{1}{\sigma_k}\right)\log\log\left(\frac{1}{\sigma_k}\right)}.
	\end{align}
\end{theorem*}
Now, one would like to merge the index bound of Lazer and Solimini with Struwe's monotonicity trick, which requires a new argument (we refer to Section \ref{min-max} for the definitions of index, nullity and of the different types of min-max families).

\begin{theorem}\label{indexbound}
	Let $(X,\norm{\,\cdot\,}_X)$ be a $C^2$ Finsler manifold modelled on a Banach space $E$, and $Y\hookrightarrow X$ be a $C^2$ Finsler-Hilbert manifold modelled on a Hilbert space $H$ which we suppose locally Lipschitz embedded in $X$, 
	and let $F,G\in C^2(X,\R_+)$ be two fixed functions. Define for all $\sigma>0$, $F_\sigma=F+\sigma^2G\in C^2(X,\R_+)$ and suppose that the following conditions hold.
	\begin{enumerate}
		\item[\emph{(1)}] \emph{\textbf{Palais-Smale condition:}} For all $\sigma>0$, the function $F_{\sigma}:X\rightarrow Y$ satisfies the Palais-Smale condition at all positive level $c>0$.
		\item[\emph{(2)}] \emph{\textbf{Energy bound:}} The following energy bound condition holds : for all $\sigma>0$ and for all $\ens{x_k}_{k\in \N}\subset X$ such that 
		\begin{align*}
			\sup_{k\in \N}F_{\sigma}(x_k)<\infty,
		\end{align*}
		we have
		\begin{align*}
			\sup_{k\in \N}\norm{\D G(x_k)}<\infty.
	\end{align*}
	\item[\emph{(3)}] \textbf{\emph{Fredholm property:}} For all $\sigma>0$ and for all $x\in K(F_{\sigma})$, we have $x\in Y$, and the second derivative $D^2F_{\sigma}(x):T_xX\rightarrow T_x^{\ast}X$ restrict on the Hilbert space $T_xY$ such that the linear map $\D^2F_{\sigma}(x)\in \mathscr{L}(T_yY)$ defined by
	\begin{align*}
	D^2F_{\sigma}(x)(v,v)=\s{\D^2F_{\sigma}(x)v}{v}_{Y,x},\quad \text{for all}\;\, v\in T_xY,
	\end{align*}
	is a \emph{Fredholm} operator, and the embedding $T_xY\hookrightarrow T_xX$ is dense for the Finsler  norm $\norm{\,\cdot\,}_{X,x}$.
	\end{enumerate}
	Now, let $\mathscr{A}$ \emph{(}resp. $\mathscr{A}^{\ast}$, resp. $\bar{\mathscr{A}}$, resp. $\mathscr{A}(\alpha_{\ast})$, resp. $\mathscr{A}(\alpha^{\ast})$, where the last two families depend respectively on a homology class $\alpha_{\ast}\in H_d(Y,B)$ - where $B\subset Y$ is a fixed compact subset -  and a cohomology class $\alpha^{\ast}\in H^d  (Y)$\emph{)} be a $d$-dimensional admissible family of $Y$ (resp. a $d$-dimension dual family to $\mathscr{A}$, resp. a $d$-dimensional co-dual family to $\mathscr{A}$, resp. a $d$-dimensional homological family, resp. a $d$-dimensional co-homological family) with boundary $\ens{C_i}_{i\in I}\subset Y$.
	Define for all $\sigma>0$
	\begin{align*}
	\begin{alignedat}{3}
	&\beta(\sigma)=\inf_{A\in\mathscr{A}}\sup F_\sigma(A)<\infty,\quad &&
	\beta^{\ast}(\sigma)=\inf_{A\in \mathscr{A}^{\ast}}\sup F_{\sigma}(A),
	\quad&&\widetilde{\beta}(\sigma)=\inf_{A\in \widetilde{\mathscr{A}}}\sup F_{\sigma}(A)\\
	&\bar{\beta}(\sigma)=\inf_{A\in \mathscr{A}(\alpha_{\ast})}\sup F_{\sigma}(A),\quad&& 
	\underline{\beta}(\sigma)=\inf_{A\in \mathscr{A}(\alpha^{\ast})}\sup F_{\sigma}(A).
	\end{alignedat}
	\end{align*} 
	Assuming that the min-max value is non-trivial, \textit{i.e.}
	\begin{enumerate}\label{non-trivial}
		\item[\emph{(4)}] \textbf{\emph{Non-trivialilty:}} $\displaystyle \beta_0=\inf_{A\in \mathscr{A}}\sup F(A)>\sup_{i\in I} \sup F(C_i)=\widehat{\beta}_0$, 
	\end{enumerate}
	there exists a sequence $\ens{\sigma_k}_{k\in \N}\subset (0,\infty)$ such that $\sigma_k\conv{k\rightarrow \infty}0$, and for all $k\in \N$, there exists a critical point $x_k\in K(F_{\sigma_k})\in \mathscr{E}(\sigma_k)$ \emph{(}resp. $x_k^{\ast},\widetilde{x}_k,\bar{x}_{k},\underline{x}_k\in \mathscr{E}(\sigma_k)$\emph{)} of $F_{\sigma_k}$ satisfying the entropy condition \eqref{boltzmann} and such that respectively
	\begin{align*}
	\left\{
	\begin{alignedat}{2}
	&F_{\sigma_k}(x_k)=\beta(\sigma_k),\quad&&
	\mathrm{Ind}_{F_{\sigma_k}}(x_k)\leq d\\
	&F_{\sigma_k}(x_k^{\ast})=\beta^{\ast}(\sigma_k),\quad&&
	\mathrm{Ind}_{F_{\sigma_k}}(x_{k})\geq d		\\
	&F_{\sigma_k}(\widetilde{x}_k)=\widetilde{\beta}(\sigma_k),\quad&&
	\mathrm{Ind}_{F_{\sigma_k}}(\widetilde{x}_k)\leq d\leq \mathrm{Ind}_{F_{\sigma_k}}(\widetilde{x}_k)+\mathrm{Null}_{F_{\sigma_k}}(\widetilde{x}_k)
	\\
	&F_{\sigma_k}(\bar{x}_k)=\bar{\beta}(\sigma_k),\quad&&
	\mathrm{Ind}_{F_{\sigma_k}}(\bar{x}_k)\leq d\leq \mathrm{Ind}_{F_{\sigma_k}}(\bar{x}_k)+\mathrm{Null}_{F_{\sigma_k}}(\bar{x}_k)
	\\
	&F_{\sigma_k}(\underline{x}_k)=\underline{\beta}(\sigma_k),\quad&&
	\mathrm{Ind}_{F_{\sigma_k}}(\underline{x}_k)\leq d\leq \mathrm{Ind}_{F_{\sigma_k}}(\underline{x}_k)+\mathrm{Null}_{F_{\sigma_k}}(\underline{x}_k).
	\end{alignedat}\right.
	\end{align*}
\end{theorem}

\begin{rem}
	The previous theorem is stated for a family $F_{\sigma}=F+\sigma^2G$, but it would hold more generally under the hypothesis of the previous Theorem $(\ast)$. Notice that the \textbf{Energy Bound} is nothing else that the bound of Theorem $(\ast)$. 
\end{rem}

\begin{rems}(On  the optimality of the hypothesis of Theorem \ref{indexbound}.)
	
	Firstly, the \textbf{Palais-Smale condition} might be weakened to the Palais-Smale condition along certain near-optimal sequence (see \cite{ghoussoub2}). However, the sequence $\ens{\sigma_k}_{k\in \N}$ given by the theorem cannot be made explicit, as is depends on differentiability property of $\sigma\mapsto\beta(\sigma)$ (actually, of certain approximations of this function), a function which is \emph{a priori} impossible to determine explicitly for all $\sigma> 0$ (determining $\beta(0)$ is already a very non-trivial question in many examples, and is actually one of the motivations of the viscosity method), so hypothesis $(1)$ is nearly optimal. 
	
	Secondly, the \textbf{Energy bound} is a mere restatement of inequality \eqref{pseudo}, which is really necessary to make the pseudo-gradient argument work (see \cite{geodesics}). It seems to be essentially the only way to obtain Palais-Smale min-max principle.
	
	Thirdly, the restriction on the Hilbert space is used to take advantage of the Morse lemma, a necessary tool in all classical references (\cite{lazer}, \cite{viterbo}, \cite{solimini}, \cite{ghoussoub}, \cite{ghoussoub2}). The \textbf{Fredholm property} is probably necessary as all existing methods rely on perturbation methods using the Sard-Smale theorem (\cite{smale}), for which the Fredholm hypothesis is necessary, thanks of the counter-example of Kupka (\cite{kupka}). Furthermore, we have to make the hypothesis that $T_xY$ be dense in $T_xX$ for a critical point $x\in K(F_{\sigma})$ as it shows that the index does not change for the restriction $\D^2F_{\sigma}(x)\in \mathscr{L}(T_xY)$.
	
	Finally, the \textbf{Non-triviality} assumption is to our knowledge necessary. Indeed, as we cannot localise the critical points of the right index as in the works of Solimini (\cite{solimini}) and Ghoussoub (\cite{ghoussoub}, \cite{ghoussoub2}), the corresponding theorem is Corollary $10.5$ in \cite{ghoussoub}, where this hypothesis is made in order to make sure that one can apply the deformation lemma. Once again, this step is the same that permits to prove the Palais-Smale min-max principle.
\end{rems}

\subsection{Examples of admissible families}

	We remark that the different families introduced above allow one to recover all known types of min-max considered by Palais in \cite{palais2}. The only case to check are the homotopy classes of mappings.
	 Let $M^d$ be a smooth manifold and let $c$ a regular homotopy class of immersions of $\Sigma^k$ into $X$, or an isotopy class of embeddings of $M^d$ into $X$. Then
	\begin{align*}
		\mathscr{A}(c)=\ens{f(M^d): f\in c}
	\end{align*}
	is ambient isotopy invariant so is an admissible family of dimension $d$, \textit{i.e.} one may freely has additional constraints in the definition of the admissible families as long as they stable under homeomorphisms isotopic to the identity (preserving the boundary conditions, if any). In particular, if $\Sigma^k, N^n$ are two smooth manifolds, $\mathrm{Imm}(\Sigma^k,N^n)$ is the set of \emph{smooth} immersions from $\Sigma^k$ to $N^n$, and $d\in \N$ is such that
	\begin{align*}
		\pi_d\left(\mathrm{Imm}(\Sigma^k,N^n)\right)\neq \ens{0},
	\end{align*}
	where $\pi_d$ designs the $d$-th regular homotopy group, then for all $c\in \pi_d\left(\mathrm{Imm}(\Sigma^k,N^n)\right)$ with $c\neq 0$, and for all $l\in \N$ and $1\leq p<\infty$ such that $lp>k$, as the following Sobolev space of immersion is a smooth Banach manifold (\cite{rivcours})
	\begin{align*}
		\mathrm{Imm}_{l,p}(\Sigma^k,N^n)=W^{l,p}(\Sigma^k,N^n)\cap\ens{\phi:d\phi(p)\wedge d\phi(p)\neq 0\;\,\text{for all}\;\, p\in \Sigma^k},
	\end{align*}
	we deduce that
	\begin{align*}
		\mathscr{A}(c)=\mathscr{P}(\mathrm{Imm}_{l,p}(\Sigma^k,N^n))\cap\ens{\phi(S^d): \phi\in C^0(S^d,\mathrm{Imm}_{l,p}(\Sigma^k,N^n)), [\phi]=c}
	\end{align*}
	is a $d$-dimensional min-max family of $\mathrm{Imm}_{l,p}(\Sigma^k,N^n)$.

\subsection{Applications}\label{appli}

\textbf{Sacks-Uhlenbeck $\alpha$-energies (\cite{sacks}).} Let $\Sigma$ be a closed Riemann surfaces  and let $(M^n,h)$ be a closed Riemannian manifold which we suppose isometrically embedded in some Euclidean space $\R^N$, and define for all $\sigma\geq 0$ the family of Banach spaces
\begin{align*}
	X_{\sigma}&=W^{1,1+\sigma}(\Sigma,M^n)=W^{1,1+\sigma}(\Sigma,\R^N)\cap\ens{\vec{u}: \vec{u}(p)\in M^n\;\,\text{for a.e.}\;\, p\in \Sigma}.\\
	Y&=W^{2,2}(\Sigma,M^n).
\end{align*}
One can check that also $X_{\sigma}$ depends on $\sigma$, as $Y$ is independent of $\sigma$, the proof of Theorem \ref{indexbound} is still valid. The function $F_{\sigma}:X_{\sigma}\rightarrow \R$ is given by
\begin{align*}
	F_{\sigma}(\vec{u})=\frac{1}{2}\int_{\Sigma}\left(\left(1+|d\vec{u}|^2\right)^{1+\sigma}-1\right)d\mathrm{vol}_{g_0}
\end{align*}
where $g_0$ is some fixed smooth metric on $\Sigma$.

The significance of the restriction on the Hilbert space $Y$ is given by the following regularity result (see \cite{moore}).

\begin{theorem*}
	If $0<\sigma<1/2$, any critical point $\vec{u}\in X_{\sigma}$ of $F_{\sigma}$ is smooth. Furthermore, for all $0<\sigma<1/2$, and all critical point $\vec{u}\in K(F_{\sigma})$, the restriction $D^2F_{\sigma}(\vec{u}):T_{\vec{u}}Y\rightarrow T_{\vec{u}}^{\ast}Y$ is a Fredholm operator.
\end{theorem*}

In particular, such critical point $u\in X_{\sigma}$ is an element of $Y$, and the definition of the index is unchanged, so the main Theorem \ref{indexbound} applies.

We find interesting to notice that this idea to restrict a functional defined on a Finsler manifold to a Finsler-Hilbert manifold in order to exploit standard Morse theory in infinite dimension is due to Karen Uhlenbeck (\cite{uhlenbeck}).

In order to introduce the next two categories, we introduce some additional definitions.
Let $\Sigma$ be a closed Riemann surfaces of genus $\gamma$, and $\mathrm{Diff}_+^{\ast}(\Sigma)$ be the topological group of positive $W^{3,2}$-diffeomorphism (we adopt the standard notations of \emph{e.g.} \cite{brezis} for Sobolev functions) with either $3$ distinct marked points if $\gamma=0$, or $1$ marked point for $\gamma=1$ and no mark points for higher genera. Furthermore, if $(M^n,h)$ is a fixed Riemannian manifold, we denote by $\mathrm{Imm}(\Sigma,M^n)$ the Banach manifold of $W^{k,p}$-immersions (for $kp>2$) by 
\begin{align*}
	\mathrm{Imm}_{{k,p}}(\Sigma,M^n)=W^{k,p}(\Sigma,M^n)\cap\ens{\phi:d\phi(p)\wedge d\phi(p)\neq 0\;\, \text{for all}\;\, p\in \Sigma}
\end{align*}

It was recently proved by Tristan Rivi\`{e}re (\cite{hierarchies}) that the quotient spaces 
\begin{align*}
    &X=\widetilde{\mathrm{Imm}}_{2,4}(\Sigma,M^n)=\mathrm{Imm}_{2,4}(\Sigma,M^n)/\mathrm{Diff}_+^{\ast}(\Sigma)\\
	&Y=\widetilde{\mathrm{Imm}}_{3,2}(\Sigma,M^n)=\mathrm{Imm}_{3,2}(\Sigma,M^n)/\mathrm{Diff}_+^{\ast}(\Sigma)
\end{align*}
are (respectively) separated smooth Banach and \emph{Hilbert} manifolds, and this is really a crucial fact, as by the invariance under the diffeomorphism group on $\Sigma$, the perturbed functional of the area of the Willmore energy cannot satisfy the Palais-Smale condition, but satisfies this condition on the quotient space.

\textbf{Minimal surfaces (\cite{viscosity}, \cite{multiplicity}).}
Here, the Finsler manifolds are
\begin{align*}
	X=\widetilde{\mathrm{Imm}}_{2,4}(\Sigma,M^n),\quad Y=\widetilde{\mathrm{Imm}}_{3,2}(\Sigma,M^n)
\end{align*}
and the functions
\begin{align*}
	F(\phi)=\mathrm{Area}(\phi(\Sigma))=\int_{\Sigma}d\mathrm{vol}_g,\quad G(\phi)=\int_{\Sigma}\left(1+|\vec{\I}_g|^2\right)^2d\mathrm{vol}_g
\end{align*}
where $g=\phi^{\ast}h$ is the pull-back of the metric $h$ on $M^n$ by the immersion $\phi$, and $\vec{\I}_g$ is the second fundamental form of the immersion $\phi:\Sigma\rightarrow M^n$. However, we see that the subtlety here is that $F_{\sigma}$ satisfies the Palais-Smale condition only on $X$, but not on $Y$. However, as for all critical point $\phi\in X$, we actually have $\phi\in C^{\infty}(\Sigma,M^n)$, then we have in particular $\phi\in Y$, and one can directly verify that $D^2F_{\sigma}(\phi)$ is Fredholm on the \emph{Hilbert} space $T_{\phi}Y$ (see \cite{hierarchies}). Therefore, the main Theorem \ref{indexbound} applies to the viscosity method for minimal surfaces. Combining the recent resolution of the multiplicity one conjecture proved in this setting by A. Pigati and T. Rivi\`{e}re (\cite{multiplicity}) with the previous result of T. Rivi\`{e}re (\cite{hierarchies}), one can obtain the lower semi-continuity of the index.

\textbf{Willmore surfaces (\cite{eversion}, \cite{classification}).}

The goal here is to go further the minimisation for Willmore surfaces in space forms and to show the existence of Willmore surfaces solution to min-max problems, such as the so-called min-max sphere eversion (\cite{kusner}).

Restricting to the special case of Willmore spheres, we take
\begin{align*}
	X=\widetilde{\mathrm{Imm}}_{2,4}(S^2,\R^n),\quad Y=\widetilde{\mathrm{Imm}}_{3,2}(S^2,\R^n)
\end{align*}
and
\begin{align*}
	F(\phi)=\mathrm{W}(\phi)=\int_{S^2}|\H_g|^2d\vg,\quad F_{\sigma}(\phi)=F(\phi)+\sigma^2\int_{S^2}\left(1+|\vec{H}_g|^2\right)^2d\vg+\frac{1}{\log\left(\frac{1}{\sigma}\right)}\mathscr{O}(\phi)
\end{align*}
where $\H_g$ is the mean-curvature of the immersion $\phi:S^2\rightarrow \R^n$, and $\mathscr{O}(\phi)$ is the Onofri energy, defined by
\begin{align*}
	\mathscr{O}(\phi)=\frac{1}{2}\int_{S^2}|d\alpha|_g^2d\vg+4\pi\int_{S^2}\alpha e^{-2\alpha}d\mathrm{vol}_g-2\pi \log\left(\int_{S^2}d\vg\right)\geq 0.
\end{align*}
where $\alpha:S^2\rightarrow \R$ is the function given by the Uniformisation Theorem such that $g=e^{2\alpha}g_0$, where $g_0$ is a fixed metric on $S^2$ of constant Gauss curvature independent of $g$. That this quantity is non-negative was proved by Onofri (\cite{onofri}). Here one also easily proves that the hypothesis of the main Theorem \ref{indexbound} are satisfied.

For a proof of the lower semi-continuity of the index and an explicit application, we refer to \cite{morse2}.

\textbf{Acknowledgements.} I would like to thank my advisor Tristan Rivi\`{e}re for his support and very 
interesting related discussions. I also wish to thank Alessandro Pigati for critically listening a preliminary version of this article. 

\subsection{Organisation of the paper}

As is fairly standard in this theory, the proof is divided into two steps between the non-degenerate case and the degenerate case. In the first one, we assume that for all $\sigma>0$, the approximation  $F_{\sigma}$ is non-degenerate and in the second step that $\D^2F(x):T_xX\rightarrow T_xX$ is a Fredholm map at every critical point $x\in K(F_{\sigma})$. Through a general perturbation method due to Marino and Prodi (\cite{marino}), it is possible to reduce the problem to the non-degenerate case, but this is quite subtle to perturb the function to preserve the entropy condition, contrary to \cite{lazer} where the degenerate case followed directly from the non-degenerate case. 

Furthermore, let us emphasize that there is to our knowledge no method to prove directly Morse index estimates in this setting without reducing to the non-degenerate case, and the Fredholm hypothesis on the second derivative becomes at this point necessary as the only known way to perturb a function on a Finsler-Hilbert manifold to make it non-degenerate is to use the Sard-Smale theorem, for which this hypothesis is necessary.

\section{Technical lemmas}

\subsection{Preliminary definitions}

\begin{defi}
	Let $\pi:\mathscr{E}\rightarrow X$ be a Banach space bundle over a Banach manifold $X$ and let $\norm{\,\cdot\,}:\mathscr{E}\rightarrow \R_+$ be a continuous function such that for all $x\in X$ the restriction $\norm{\,\cdot\,}_{x}$ is a norm on the fibre $\mathscr{E}_x=\pi^{-1}(\ens{x})$. For all $x_0\in X$, and for all trivialisation $\varphi_{x_0}:\pi^{-1}(U_{x_0})\rightarrow U_{x_0}\times \mathscr{E}_{x_0}$ (where $U$ is some open neighbourhood of $x_0$) then for all $x\in U$, we get an isomorphism $L_x:\mathscr{E}_{x}\rightarrow \mathscr{E}_{x_0}$ so $\norm{\,\cdot\,}_x$ induces a norm on $\mathscr{E}_{x_0}$ by
	\begin{align*}
	\norm{v}_{x}=\norm{L_x^{-1}(v)}_{x}\qquad(\text{for all}\;\,v\in \mathscr{E}_{x_0}).
	\end{align*}
	We say that $\norm{\,\cdot\,}:\mathscr{E}\rightarrow \R$ is a Finsler structure on $\mathscr{E}$ is for all $x_0\in X$ and all such trivialisation $(U_{x_0},\varphi_{x_0})$, there exists a constant $C=C_{x_0}\geq 1$ such that for all $x\in U_{x_0}$, 
	\begin{align*}
	\frac{1}{C}\norm{\,\cdot\,}_x\leq \norm{\,\cdot\,}_{x_0}\leq C\norm{\,\cdot\,}_x.
	\end{align*}
	A Finsler manifold is a \emph{regular} (in the topological sense) $C^1$ Banach manifold $X$ equipped with a Finsler structure on the tangent space $TX$. A Finsler-Hilbert manifold or (infinite-dimensional) Riemannian manifold is a Finsler manifold modelled on a Hilbert space.
\end{defi}

\begin{theorem}[Palais \cite{palais2}\label{palais}]
	Let $(X,\norm{\,\cdot\,})$ be a Finsler manifold, and $d:X\times X\rightarrow \R_+\cup\ens{\infty}$ be such that for all $x,y\in X$
	\begin{align*}
	d(x,y)=\inf\ens{\int_{0}^{1}\norm{\gamma'(t)}_{\gamma(t)}dt: \gamma\in C^0([0,1],X), \gamma(0)=x,\;\, \gamma(1)=y}.
	\end{align*}
	Then $d$ is a distance on $X$ inducing the same topology as the manifold topology on  $X$.
\end{theorem}

In particular, we will always assume that Finsler manifolds equipped with their Palais distance, usually denoted by $d$, and we will denote for all $A\subset X$ and $\delta>0$
\begin{align*}
&N_{\delta}(A)=X\cap\ens{x: d(x,A)\leq \delta}\\
&U_{\delta}(A)=X\cap\ens{x: d(x,A)<\delta}.
\end{align*}
\begin{theorem}[Palais, \cite{palais2}]\label{palais2}
	Let $(X,\norm{\,\cdot\,})$ be a Finsler manifold modelled on some Banach space $E$, let $U\subset X$ be an open subset, $\varphi:U\rightarrow E$ a chart and $x_0\in U$. We define for all $r>0$
	\begin{align*}
	&B(x_0,r)=U\cap\ens{x:\norm{\varphi(x)-\varphi(x_0)}\leq r}\\
	&U(x_0,r)=U\cap\ens{x: \norm{\varphi(x)-\varphi(x_0)}<r}\\
	&S(x_0,r)=U\cap\ens{x: \norm{\varphi(x)-\varphi(x_0)}=r}.
	\end{align*}
	Then for $r>0$ sufficiently small $B(x_0,r)$ is a closed neighbourhood of $x_0$, $U(x_0,r)$ is its interior relative to $X$ and $S(x_0,r)$ is the frontier relative to $X$.
\end{theorem}

\begin{cor}\label{nottrivial}
	Let $(X,\norm{\,\cdot\,})$ be a Finsler manifold and $K\subset X$ be a compact  subset. Then for $r>0$ small enough, $N_{\delta}(K)$ is closed, and $U_{\delta}(K)$ is its interior relative to $X$.
\end{cor}

\begin{defi} 
	Let $E,F$ be two Banach spaces. We say that a linear map $T\in\mathscr{L}(E,F)$  is a \emph{Fredholm operator} if $\Im(T)\subset F$ is closed, and  $\mathrm{Ker}(T)\subset E$ and $\mathrm{Coker}(T)=F/\mathrm{Im}(T)$ are finite-dimensional. Then the index $\mathrm{Ind}(T)\in \Z$ is defined by
	\begin{align*}
	\mathrm{Ind}(T)=\dim\mathrm{Ker}(T)-\mathrm{dim}(\mathrm{Coker}(T)).
	\end{align*}
\end{defi}

\begin{defi}
	Let $X,Y$ be two Banach manifolds and $F:X\rightarrow Y$ be a $C^1$ map. We say that $F$ is a \emph{Fredholm map} at $x$ if $DF(x):T_xX\rightarrow T_{F(x)}Y$ is a Fredholm operator and we define the index of $F$ at $x$, still denoted by $\mathrm{Ind}_x(F)$, by
	\begin{align*}
	\mathrm{Ind}_x(F)=\mathrm{Ind}(DF(x)).
	\end{align*}
\end{defi}

As the map $x\mapsto \mathrm{Ind}_x(F)\in \Z$ is continuous, we deduce that it is constant on each connected component of $X$, and we will denote it by $\mathrm{Ind}(F)$ if $F$ is defined on a connected domain.

In the applications we have in mind, we cannot assume that the manifold $X$ be connected, so we will have to keep in mind this technical point. 

If $F:X\rightarrow Y$ is a $C^1$ map between Banach manifolds, we say that $x\in X$ is a regular point if $DF(x):X\rightarrow Y$ is surjective. The complement of the regular points are called the \emph{singular points}, the image under $F$ of the singular points are the \emph{critical values} and their complement the \emph{regular values}.

Now we recall the celebrated Sard's theorem of Smale, which proceeds by reducing the infinite dimensional version to the finite dimensional Sard's theorem.

\begin{theorem}[Smale, \cite{sardsmale}]\label{sard}
	Let $X,Y$ be two Banach manifolds and let $U\subset X$ be an open \emph{connected} subset and $F:U\rightarrow Y$ be a $C^q$ Fredholm map, where
	\begin{align*}
	q\geq \max\ens{\mathrm{Ind}(F),0}+1.
	\end{align*}
	Then the regular values of $F$ are almost all $Y$, \textit{i.e.} the set of critical value is a set of first Baire category (or meagre).
\end{theorem}

\subsection{Morse Index and admissible families of min-max}\label{min-max}

Let $X$ a $C^2$ Banach manifold, and suppose that $F:X\rightarrow\R$ is a function which admits second order G\^{a}teaux derivatives in $X$, \textit{i.e.} for all $C^2$ path $\gamma:(-\epsilon,\epsilon)\rightarrow X$ the function $t\mapsto F(\gamma(t))$ is a $C^2$ function. Then a critical point $x\in X$ of $F$ is an element such that for all $C^2$ path $\gamma:(-\epsilon,\epsilon)\rightarrow X$ such that $\gamma(0)=x$, we have
\begin{align*}
	\frac{d}{dt}F(\gamma(t))_{|t=0}=0.
\end{align*}
If $x$ is a critical point, we define the second derivative quadratic form $Q_x=D^2F(x):T_xX\rightarrow (T_xX)^{\ast}$ by
\begin{align*}
	Q_x(v)(v)=D^2F(x)(v,v)=\frac{d^2}{dt^2}F(\gamma(t))_{|t=0}
\end{align*}
for all $v\in T_xX$ and path $\gamma:(-\epsilon,\epsilon)\rightarrow X$ such that $\gamma(0)=x$ and $\gamma'(0)=v$.

Then $Q_x$ is a well-defined continuous map on $T_xX$, and the index $\mathrm{Ind}_{F}(x)$ of $x$ with respect to $F$, is defined by
\begin{align*}
	\mathrm{Ind}_{F}(x)=\sup\ens{\dim V:V\subset T_xX\;\,\text{is a sub vector-space such that}\;\, Q_x(v)(v)<0\;\,\text{for all}\;\, v\in V}
\end{align*}
To define the nullity, we need to assume that $F:X\rightarrow \R$ is $C^2$ Fr\'{e}chet differentiable map and recalling that $Q_x=D^2F(x):T_xX\rightarrow (T_xX)^{\ast}$, we define
\begin{align*}
	\mathrm{Null}_F(x)=\sup \ens{\dim W:W\subset T_xX\;\, \text{is a sub vector-space such that}\;\, Q_x(w)=0\;\,\text{for all}\;\, w\in W}.
\end{align*}
If $F$ is more regular or $X$ is a Finsler-Hilbert manifold, the definition remains unchanged. That is, if $X$ is a Finsler-Hilbert manifold, then we have
\begin{align*}
	D^2F(x)(v,v)=\s{L v}{v}_x
\end{align*}
for some self-adjoint linear operator $L:T_xX\rightarrow T_xX$. Its number of negative eigenvalues (with multiplicity) is also equal to the index of $F$ by the preceding definition (while the nullity is equal to the number of Jacobi fields, \textit{i.e.} $\mathrm{Null}_F(x)=\dim \mathrm{Ker} (L)$).

\begin{remimp}
	In particular, if $Y\subset X$ is a Lipschitz embedded \emph{Hilbert} manifold, and $x\in Y$ is a critical point of $F$, then the square gradient $\D^2F(x):T_{x}X\rightarrow T_{x}X$ restricts continuously to the \emph{Hilbert} space $T_{x}Y$ and the definition of the index is unchanged, provided that $T_xY\subset T_{x}X$ is dense, a condition easily verified in the cases of interest to us (it is stated explicitly in the hypothesis od the main Theorem \ref{indexbound}).
\end{remimp}

We first define families of min-max based on families of continuous maps.

\begin{defi}[Min-max families]\label{defminmax} Let $X$ be a $C^1$ Finsler manifold.	

\textbf{(1)} \textbf{ Admissible family.}	We say that $\mathscr{A}\subset \mathscr{P}(X)\setminus\ens{\varnothing}$ is an admissible min-max family of dimension $d\in\N$ with boundary $(B^{d-1},h)$ (possibly empty)
	for $X$ if
	\begin{enumerate}
		\itemsep-0.2em 
		\item[(A1)] For all $A\in\A$, $A$ is compact in $X$,
		\item[(A2)] There exists a $d$-dimensional compact \emph{Lipschitz} manifold $M^d$ with boundary $B^{d-1}$, (possibly empty)
		 and a continuous map $h:B\rightarrow X$ such that for all $A\in\mathscr{A}$, there exists a continuous map $f:M^d\rightarrow X$ that $A=f(M^d)$, and $f=h$ on $B$. 
		\item[(A3)] For every homeomorphism $\varphi$ of $X$ isotopic to the identity map such that $\varphi|_{B}=\mathrm{Id}|_{h(B)}$, and for all $A\in\A$, we have $\varphi(A)\in A$.
	\end{enumerate}

More generally, one can relax the notions of uniqueness of the compact manifold $M^d$ as follows. Let $I$ a set of indices and a family $\ens{M_i^d}_{i\in I}$ of compact Lipschitz manifold with boundary $(B^{d-1}_i,h_i)$. Then we define
\begin{align*}
	\mathscr{A}=\mathscr{P}(X)\cap \ens{A: \text{there exists}\;\,i\in I\;\,\text{and}\;\, f\in C^0(M_i^d,X)\;\, \text{such that}\;\, A=f(M_i^d)\;\,\text{and}\;\, f_{B_i^{d-1}}=h_i}
\end{align*}
Clearly, this class is stable under homeomorphisms of $X$ isotopic to the identity preserving the boundary $h(B)$.

\textbf{(2)} \textbf{Dual admissible family.} In a dual fashion, let $I$ be a (non-empty) sets of indices and let $\ens{C_i}_{i\in I}\subset X$ be a collection of subsets such that for all $i\in I$, there exists a non-empty set $J_i$ and a family of continuous functions $\{h_i^j\}_{j\in J_i}\in C^0(C_i,\R^d)$. Then we define $\mathscr{A}^{\ast}=\mathscr{A}(I,\ens{J_i}_{i\in I}, \{h_i^j\}_{i\in I, j\in J_i})$ by
\begin{align*}
	\mathscr{A}^{\ast}=\mathscr{P}(X)\cap\{A: &\text{ there exists } i\in I\;\, \text{such that for all } h\in C^0(X,\R^d)\;\, \text{such that } h_{|C_i}=h_{i}^j\;\, \text{for some } j\in I\\
	&\text{one has}\;\, 0\in h(A)
	\}.
\end{align*}
If the functions $h_i:B^{d-1}_i\rightarrow X$ are implicit, then we say by abuse of notation that $\ens{C_i}_{i\in I}=\ens{h(B_i^{d-1})}_{i\in I}$ is the boundary of $\mathscr{A}$ (this permits to give a uniform definition of boundary for each of admissible families).

\textbf{(3)} \textbf{Co-dual admissible family.} Finally, given a $d$-dimension dual admissible family $\mathscr{A}^{\ast}$, a $d$-dimensional co-dual admissible family is defined by
\begin{align*}
	\widetilde{\mathscr{A}}=\mathscr{A}^{\ast}\cap\ens{A:\dim_{\mathscr{H}}(A)<d+1},
\end{align*}
where $\dim_{\mathscr{H}}$ designs the Hausdorff dimension relative to the Hausdorff measures of the metric space $X$ (equipped with its Palais distance). 
The class is only stable under locally Lipschitz homeomorphism of $X$ isotopic to the identity (this is not restrictive, as the only homeomorphisms of interest are gradient flow of $C^2$ functions, which are indeed locally Lipschitz).

Finally, we define the following boundary values of admissible families $\mathscr{A}$, $\mathscr{A}^{\ast}$ and $\widetilde{\mathscr{A}}$ with boundary $\ens{C_i}_{i\in I}$ by 
\begin{align*}
	&\widehat{\beta}(F,\mathscr{A})=\sup_{i\in I}\,\sup F(h_i(B_i^{d-1})),\quad (\text{where}\;\, C_i=h(B^{d-1}_i)\;\, \text{for all}\;\, i\in I)\\
	&\widehat{\beta}(F,\mathscr{A}^{\ast})=\widehat{\beta}(F,\widetilde{\mathscr{A}})=\sup_{i\in I}\,\sup F(C_i).
	.
\end{align*}
\end{defi}

\begin{rem}
	The definition of the third family in \cite{lazer} is the more restrictive
	\begin{align*}
		\widetilde{\mathscr{A}}=\mathscr{A}^{\ast}\cap\ens{A: \mathscr{H}^d(A)<\infty}
	\end{align*}
	but as we shall see, our definition will still permit to obtain the suitable two-sided index bounds.
\end{rem}

\begin{rem}
	In the definition of the first family of min-max, the hypothesis on $M^d$ (or equivalently on $\ens{M_i^d}_{i\in I}$) can be considerably weakened, as the main Theorem \ref{indexbound} would still hold if $M^d$ was merely a metric space of Hausdorff dimension (with respect to the metric) at most $d$ admitting Lipschitz partitions of unity. Furthermore, the family of boundaries $\ens{B_i^{d-1}}_{i\in I}$ need not be a boundary, but can be any closed subset, as long as it satisfies the non-triviality condition as recalled below. In particular, $M^d$ can be assumed to be a cellular complex of dimension at most $d$. This will be particularly important in the example of Section \ref{tristan}, where we shall also in some special situation relax the hypothesis relative to the continuity of the different functions involved in a situation where weaker topologies are available.
\end{rem}

\begin{defi}
	Let $X$ be a $C^1$ Finsler manifold and $\mathscr{A}$ (resp. $\mathscr{A}^{\ast}$, resp. $\widetilde{\mathscr{A}}$) be a $d$-dimensional admissible (resp. dual, resp. co-dual) min-max family  with boundary $\ens{C_i}_{i\in I}$. We say that $\mathscr{A}$ (resp. $\mathscr{A}^{\ast}$, resp. $\widetilde{\mathscr{A}}$) is \emph{non-trivial} with respect to a continuous map $F:X\rightarrow \R$ if 
	\begin{align}\label{adapted}
		\beta(F,\mathscr{A})=\inf_{A\in \mathscr{A}}\sup F(A)>\sup \sup F(h_i(B_i^{d-1}))=\widehat{\beta}(F,\mathscr{A}).
	\end{align}
	Whenever this does not yield confusion, we shall write more simply $\beta(\mathscr{A})$ and $\widehat{\beta}(\mathscr{A})$.
\end{defi}

\begin{rem}
	The condition $\mathrm{(A2)} 
	$ can be relaxed in the sense that the applications $f:M^d\rightarrow X$ need not be continuous with respect to the strong topology of $X$, as long as we take a weaker notion of continuity stable under homeomorphisms of $X$ isotopic to the identity and fixing the boundary $h(B)$. See Section \ref{tristan} for an explicit example involving families of immersions continuous with respect to the flat norm of currents.
\end{rem}

The second class of mappings are based on (co)-homology type properties.

\begin{defi}
	Let $R$ be an arbitrary ring,  $G$ be an abelian group, and $d\in \N$ a fixed integer.
	
	\textbf{(4)} \textbf{Homological family.} Let ${\alpha}_{\ast}\in H_d(X,B,R)\setminus \ens{0}$ be a \emph{non-trivial} $d$-dimensional relative (singular) homology class of $X$ with respect to $B$ with $R$ coefficients. We say that $\underline{\mathscr{A}}=\underline{\mathscr{A}}(\alpha_{\ast})$ is a $d$-dimensional homological family with respect to $\alpha_{\ast}\in H_d(X,B,R)$ and boundary $B$ if
	\begin{align*}
		\underline{\mathscr{A}}({\alpha}_{\ast})=\mathscr{P}(X)\cap \ens{A: A\;\, \text{compact}, B\subset A\;\, \text{and}\;\, \alpha\in \mathrm{Im}(\iota^{A}_{\ast})},
	\end{align*} 
	where for all $A \supset B$, the application $\iota_{\ast}^A: H_d(A,B,R)\rightarrow H_d(X,A,R)$ is the induced map in homology from the injection $\iota^{A}:A\rightarrow X$.
	
	\textbf{(5)} \textbf{Cohomological family.} Let $\alpha^{\ast}\in H^{d}(X,G)\setminus\ens{0}$ be a \emph{non-trivial} $d$-dimensional (singular) cohomology class of $X$ with $G$ coefficients. We say that $\bar{\mathscr{A}}=\bar{\mathscr{A}}(\alpha^{\ast})$ is a $d$-dimensional cohomological family with respect to $\alpha^{\ast}\in H^d(X,G)$ if
	\begin{align*}
		\bar{\mathscr{A}}(\alpha^{\ast})=\mathscr{P}(X)\cap\ens{A: A \text{ compact and}\;\, \alpha^{\ast}\notin \mathrm{Ker}(\iota^{\ast}_{A})},
	\end{align*}
	where for all $A\subset X$, the application $\iota^{\ast}_{ A}:H^d(X,G)\rightarrow H^d(A,G)$ is the induced map in cohomology from the injection $\iota_{A}:A\rightarrow X$. In other word, the non-zero class $\alpha^{\ast}$ is not annihilated by the restriction map in cohomology $\iota^{\ast}_{ A}:H^d(X,G)\rightarrow H^d(A,G)$.

\end{defi}

\begin{rem}
	This recovers the classes (e) and (f) in the seminal paper of Palais (\cite{palais}).
	We observe that for cohomological families, there is no boundary conditions to check, as they are obviously stable under any ambient homeomorphism isotopic to the identity $\mathrm{Id}_{X}:X\rightarrow X$. One can check that no restrictions is necessary for the coefficients in homology and cohomology.
\end{rem}

\subsection{Deformation lemmas}

The results we present here are essentially adaptations to our setting of known results of Lazer-Solimini and Solimini (see also the results of Ghoussoub for subsequent extensions \cite{ghoussoub}, \cite{ghoussoub2}).

The next lemma is due to Solimini and absolutely crucial as, whereas the restriction of $F_{\sigma}$ on the Hilbert does \emph{not} satisfy the Palais-Smale condition, it satisfies a stronger property on a suitable neighbourhood of critical points.

If $(X,\norm{\,\cdot\,})$ is a Finsler manifold equipped with its Palais distance $d$ and $A\subset X$, we recall the notations
\begin{align*}
&U_{\delta}(A)=X\cap\ens{x: d(x,A)<\delta},\qquad 
N_{\delta}(A)=X\cap\ens{x:d(x,A)\leq \delta}.
\end{align*}
Notice in particular that by Corollary \ref{nottrivial}, if $A$ is assumed to be compact, then $N_{\delta}(A)$ is closed and $U_{\delta}(A)$ is its interior. In all constructions, we will assume implicitly whenever necessary that such $\delta>0$ has been chosen such that $N_{\delta}(A)$ is closed.

\begin{prop}[Solimini, \cite{solimini}]\label{propre}
	Let $X$ be a $C^2$ Finsler-Hilbert manifold and $F:X\rightarrow \R$ be a $C^2$ function, and assume that $K\subset K(F)$ is a compact subset of the critical points of $F$. If the square gradient $\D^2F(x):T_xX\rightarrow T_xX$ is a Fredholm operator for all $x\in X$, for all $\epsilon>0$, there exists $\delta>0$ such that for all $\tilde{F}:X\rightarrow \R$ such that
	\begin{align*}
		\Vert F-\tilde{F}\Vert_{C^2}\leq \epsilon
	\end{align*}
	the map $\D\tilde{F}$ is proper on $N_{\delta}(K)$. In particular, $\tilde{F}$ satisfies the Palais-Smale condition on $N_{\delta}(K)$.
\end{prop}

\begin{rem}
	That $D F$ is proper near critical points $x\in X$ where $D^2F(x)$ is Fredholm is a well-known property due to Smale (\cite{sardsmale}).
\end{rem}
\begin{proof}
	We first treat the case $K=\ens{x_0}$. By a remark which will be repeatedly used, we can assume by Henderson's theorem (\cite{henderson}) that $X$ is a open subset of a Hilbert space $H$. We fix some $\epsilon>0$, and we take $\delta>0$ small enough such that $\D F-\D^2F(x):X\rightarrow H$ be Lipschitz on $N_{\delta}(x_0)$ with
	\begin{align}\label{lip1}
		\mathrm{Lip}\left((\D F-\D^2F(x))|N_{\delta}(\ens{x_0})\right)\leq \epsilon,
	\end{align}
	and define $G:X\rightarrow H$ by 
	\begin{align*}
		G(x)=\D \tilde{F}(x)-\D^2 F(x_0)(x).
	\end{align*}
	Then $G$ is Lipschitz and satisfies by \eqref{lip1}
	\begin{align}\label{lip2}
		\mathrm{Lip}(G|N_{\delta}(x_0))\leq 2\epsilon.
	\end{align}
	As $D^2 F(x_0)$ is a Fredholm operator, there exists a finite dimensional vector $H_0=\mathrm{Ker}(\D^2 F(x))\subset H$ such that we have the direct sum decomposition
	\begin{align}\label{direct}
		H=H_0\oplus H_0^{\perp}.
	\end{align}
	In particular, as $H_0$ is finite dimensional, there exists a positive constant $0<\alpha<\infty$ such that for all $v\in H_0^{\perp}$, there holds
	\begin{align}\label{alpha}
		\Vert \D^2 F(x_0)(v)\Vert \geq \alpha \Vert v\Vert.
	\end{align}
	Now, assume that $\ens{x_k}_{k\in \N}\subset N_{\delta}(x_0)\subset X$ is such that $\{\D \tilde{F}(x_k)\}_{k\in \N}\subset X$ converges. Writing for all $k\in \N$
	\begin{align*}
		x_k=x_k^{0}+x_k^{\perp}
	\end{align*}
	according to the direct sum decomposition \eqref{direct}, we can assume that up to subsequence $\ens{x_k^{0}}_{k\in \N}$ is convergent in $N_{\delta}(x_0)$ (which is closed by Corollary \ref{nottrivial}). Now, for all $k,l\in \N$, we have
	\begin{align*}
		\norm{x_k^{\perp}-x_l^{\perp}}&\leq \frac{1}{\alpha}\norm{ \D^2F(x_0)(x_k^{\perp}-x_l^{\perp})}=\frac{1}{\alpha}\norm{\D^2F(x_0)(x_k^{\perp}-x_l^{\perp})}\\
		&\leq \frac{1}{\alpha}\norm{G(x_k-x_l)}+\frac{1}{\alpha}\norm{\D\tilde{F}(x_k-x_l)}\\
		&\leq \frac{2\epsilon}{\alpha}\left(\norm{x_k^{\perp}-x_l^{\perp}}+\norm{x_k^0-x_l^0}\right)+\frac{1}{\alpha}\norm{\D \tilde{F}(x_k-x_l)}
	\end{align*}
	so taking $2\epsilon< \alpha$ yields
	\begin{align*}
		\norm{x_k^{\perp}-x_l^{\perp}}\leq \frac{1}{\alpha-2\epsilon}\left(2\epsilon\norm{x_k^{0}-x_l^{0}}+\norm{\D\tilde{F}(x_k)-\D\tilde{F}(x_l)}\right)\conv{k,l\rightarrow \infty}0,
	\end{align*}
	by the assumption and the previous remark. This finishes the proof of the  special case of the proposition. As $K$ is compact, there exists a uniform $\alpha$ such that \eqref{alpha} holds for all $x_0\in K$ and appropriate $H_0=H_0(x_0)$. Taking a finite covering $\ens{N_{\delta}(x_i)}_{1\leq i\leq N}$ for $\delta>0$ small enough and some elements $\ens{x_i}_{1\leq i\leq N}\subset K$, the previous proof works identically. This concludes the proof of the general case.
\end{proof}

\begin{cor}
	Let $X$ be a $C^2$ Finsler manifold, $Y\subset X$ be a locally Lipschitz embedded  Finsler-Hilbert manifold $F,G\in C^2(X,\R_+)$ and for all $0<\sigma<1$, define $F_{\sigma}=F+\sigma^2G\in C^2(X,\R_+)$.  Let $\sigma>0$ be a fixed real number and assume that $K\subset K(F_{\sigma})$ is a compact subset such that restriction $\D^2 F_{\sigma}(x):T_xX\rightarrow T_xX$ on $X$ is a Fredholm operator on a compact subset $K\subset K(F_{\sigma})$. Then for all $\epsilon>0$, there exists $\delta>0$ such that for all $\tilde{F}_{\sigma}\in C^2(X,\R)$ such that
	\begin{align*}
		\Vert F_{\sigma}-\tilde{F}_{\sigma}\Vert_{C^2} \leq \epsilon
	\end{align*}
	then $\D\tilde{F}_{\sigma}:Y\rightarrow Y$ is proper on $N_{\delta}(K)$. In particular, $\tilde{F}$ satisfies the Palais-Smale condition on $N_{\delta}(K)$.
\end{cor}

\begin{lemme}\label{cutoff}
	Let $X$ be a Finsler-Hilbert manifold and $K\subset X$ be a compact subset. Then for all small enough $\delta>0$ there exists a smooth function $\varphi:H\rightarrow [0,1]$ whose all derivatives are bounded and such that 
	\begin{align*}
	\left\{\begin{alignedat}{1}
	&\varphi(x)=1\quad\text{for all}\;\; x\in N_{\delta}(K)\\
	&\varphi(x)=0\quad \text{for all}\;\; x\in H\setminus N_{2\delta}(K).
	\end{alignedat}\right.
	\end{align*}
\end{lemme}
\begin{proof}
	As $K$ is compact, let $x_1,\cdots,x_n\in K$ be such that
	\begin{align*}
	K\subset \bigcup_{i=1}^nB\left(x_i,\frac{\delta}{2}\right).
	\end{align*}
	Taking $\delta$ small enough, we can make sure that each ball $B(x_i,\delta)$ is included in a chart domain into the fixed Hilbert space $(H,|\,\cdot\,|)$ model of $X$.
	Let $\eta\in C^{\infty}_c(\R,[0,1])$ be  such that
	\begin{align*}
	\left\{\begin{alignedat}{2}
	&\eta(t)=1\quad && \text{for}\;\, t\leq \frac{9}{4}\\
	&\eta(t)=0\quad && \text{for}\;\, t\geq 4
	\end{alignedat}\right.
	\end{align*}
	and let $\varphi_i\in C^{\infty}(X,\R)$ be defined by
	\begin{align*}
	\varphi_i(x)=\eta\left(\frac{|x-x_i|^2}{\delta^2}\right).
	\end{align*}
	Then $\varphi_i\in C^{\infty}(H,[0,1])$ verifies
	\begin{align*}
	\left\{\begin{alignedat}{2}
	&\varphi_i(x)=1&&\quad \text{for}\;\, x\in B\left(x_i,\frac{3}{2}\delta\right)\\
	&\varphi_i(x)=0&&\quad \text{for}\;\, x\in H\setminus B(x_i,2\delta).
	\end{alignedat}\right.
	\end{align*}
	Now letting $\zeta\in C^{\infty}_c(\R,[0,1])$ such
	\begin{align*}
	\left\{\begin{alignedat}{2}
	&\zeta(t)=0,&&\quad\text{for}\;\, t\leq 0\\
	&\zeta(t)=1&&\quad \text{for}\;\, t\geq 1
	\end{alignedat}\right.
	\end{align*}
	the function $\varphi\in C^{\infty}(H,[0,1])$ defined by
	\begin{align*}
	\varphi(x)=\zeta\left(\sum_{i=1}^{n}\varphi_i(x)\right),\quad x\in H
	\end{align*}
	has all the required properties.
\end{proof}

We recall the proof of the following perturbation method due to Marino and Prodi, as we will have to exploit the specific form of the perturbation in the proof of the main Theorem \ref{indexbound}.

\begin{prop}[Teorema $2.1$ \cite{marino}, Proposition $3.4$ \cite{solimini}]\label{marinoprodi} Let $k\geq 2$ and $X$ be a $C^k$ Finsler-Hilbert manifold and $F:X\rightarrow\R$ be a $C^k$ function, and assume that $K_0\subset K=K(F)$ is a compact subset of set of critical points of $F$. If the square second derivative $\D^2F(x):T_xX\rightarrow T_xX$ is a Fredholm operator for all $x\in K$, then for all $\epsilon,\delta>0$ small enough, there exists $\tilde{F}\in C^k(X,\R)$ such that
	\begin{align}\label{approx}
	\begin{alignedat}{1}
	&\Vert F-\tilde{F}\Vert_{C^k(X)}< \epsilon\\
	&F(x)=\tilde{F}(x)\;\, \text{for all}\;\, x\in X\setminus N_{2\delta}(K)
	\end{alignedat}
	\end{align}
	and the critical points of $\tilde{F}$ in $N_{\delta}(K)$ 
	are non-degenerate and finite in number. Furthermore, we can impose $\tilde{F}\leq F$ or $\tilde{F}\geq F$.
\end{prop}
\begin{proof}
	We can assume by Henderson's theorem that $X$ is an open subset of a Hilbert space $H$ with scalar product $\s{\,\cdot\,}{\,\cdot\,}$. Let $\varphi:X\rightarrow \R$ be the cut-off function of Lemma \ref{cutoff}, and define for $x_0\in N_{2\delta}(K_0)$, $y\in X$ the function $\widetilde{F}_{x_0,y}:X\rightarrow \R$ such that for all $x\in X$,
	\begin{align*}
	\widetilde{F}_{x_0,y}(x)=F(x)+\varphi(x)\s{y}{x-x_0}.
	\end{align*}
	As $K_0$ is compact, there exists $C_0=C_0(\delta)$ such that
	\begin{align}\label{eps1}
	\sup_{x\in N_{2\delta}(K)}\norm{x}\leq C_0(\delta)\;\,.
	\end{align} 
	Furthermore, thanks of the construction of Lemma \ref{cutoff}, we have for some universal constant $C_1$
	\begin{align}\label{eps2}
	\norm{\varphi}_{C^k(X)}=\norm{\varphi}_{C^k(N_{2\delta}(K_0))}\leq \frac{C_1}{\delta^k}.
	\end{align}
	Then for all $\Vert y\Vert\leq \frac{\delta^k}{C_0(\delta)C_1}\epsilon
	$, we get the the first property of \eqref{approx}. Furthermore, we have on $N_{\delta}(K)$
	\begin{align*}
	&\D\tilde{F}_{x_0,y}=\D F+y\\
	&\D^2\tilde{F}_{x_0,y}=\D^2F.
	\end{align*}
	In particular, $x\in N_{\delta}$ is a critical point of $\tilde{F}$ if $-y$ is a regular value of $\D \tilde{F}_{x_0,y}:X\rightarrow H$. Now, if we take $\delta>0$ small enough such that each connected component of $N_{\delta}(K)$ intersects $K$, we see that $DF$ is a Fredholm map on $N_{\delta}(K)$ of index $0$. Indeed, as $H$ is a Hilbert space, seeing all $x\in N_{\delta}(K)$ as a map $T=\D^2\tilde{F}_{x_0,y}(x):T_xX\simeq H\rightarrow T_xX\simeq H$ shows that $\D^2F(x)$ is a self-adjoint Fredholm operator (by the connectedness hypothesis), so it must have index $0$. 
	In particular, we can apply  the  Sard-Smale Theorem \ref{sard} if $\D F$ is only $C^1$ on $X$ to obtain an element $-y\in X$ such that
	\begin{align}\label{eps3}
	\norm{y}< \frac{\delta^k}{C_0(\delta)C_1}\epsilon
	\end{align}
	which is a regular value of $\D F_{x_0,y}$ (for all $x_0\in X$). Writing $\tilde{F}_{x_0}=\tilde{F}_{x_0,y}$, we see that for all $x_0\in N_{2\delta}(K)$, by \eqref{eps1}, \eqref{eps2} and \eqref{eps3}
	\begin{align*}
	\norm{F-\tilde{F}_{x_0,y}}_{C^k(X)}< \epsilon
	\end{align*}
	
	Now, once $y\in X$ is chosen, as $K$ is  compact, 
	\begin{align*}
	\sup_{x\in N_{2\delta}(K)}\s{y}{x}<\infty,
	\end{align*}
	and their exists $x_0\in N_{2\delta}(K)$ such that
	\begin{align*}
	\s{y}{x}\leq \s{y}{x_0}\quad \text{for all}\;\, x\in N_{2\delta}(K).
	\end{align*} 
	Taking $\tilde{F}=\tilde{F}_{x_0,y}$, we obtain $\tilde{F}\leq {F}$ and the conclusions of the Proposition (the other inequality $\tilde{F}\geq F$ is similar).	
\end{proof}

\section{Lazer-Solimni deformation theorem}

\subsection{Deformation and extension lemmas}

As a key technical lemma in \cite{lazer} contains an incorrect statement, we will check in this section that Lazer-Solimini's construction does not actually use this statement, so that their results are still valid (along with \cite{solimini}).

As we have mentioned it earlier, the basic principle to obtain index bounds is to first consider the case of non-degenerate functions. Therefore, we fix a $C^2$ Finsler-Hilbert manifold $X$ (modelled on a separated Hilbert) and a $C^2$ function $F:X\rightarrow \R$, for which we assume that $F$ satisfies the Palais-Smale condition at all level $c\in \R$, and to fix ideas, let $\mathscr{A}$ be a $d$-dimensional admissible family. We assume that $F$ is non-degenerate on the critical set $K(F,\beta_0)$ at level $\beta_0=\beta(F,\mathscr{A})$. In particular, as $F$ satisfies the Palais-Smale condition, $K(F,\beta_0)$ is compact and as $F$ is non-degenerate on $K(F,\beta_0)$, we deduce that $K(F,\beta_0)$ is composed of finitely many points, so that for some $x_1,\cdots,x_m\in X$, we have
\begin{align*}
	K(F,\beta_0)=\ens{x_1,\cdots,x_m}.
\end{align*}
Let $i\in \ens{1,\cdots,m}$ be a fixed integer. Then there exists closed subspaces $H_-,H_+\subset H$ such that up to the identification $T_{x_i}X\simeq H$, the square gradient $\D^2F(x)\in \mathscr{L}(T_xX)$ is negative definite on $H_-$ and positive definite on $H_+$. Furthermore, $H$ is the direct sum of $H_-$ and $H_+$, and for all $y\in  H=H_-\oplus H_+$, we write $y=y_-+y_+$, where $y_-\in H_-
$ and $y_+\in H_+$.

Furthermore, by the Morse lemma for $C^2$ functions (\cite{cambini}), for all $1\leq i\leq m$, there exists $\epsilon_i>0$ and a Lipschitz homeomorphism $\varphi_i:U_{\epsilon_i}(x_i)\rightarrow \varphi(U_{\epsilon_i}(x_i))\subset H$ such that $\varphi_i(x_i)=0\in H$ and for all $x\in \varphi(U_{\epsilon}(x_i))$, there holds
\begin{align}\label{morse}
	F(\varphi_i^{-1}(x))=F(x_i)+\norm{x_+}^2-\norm{x_-}^2,
\end{align}
where $\norm{\,\cdot\,}$ is the norm of the Hilbert space $H$. In order to make the notations lighter, we will remove most explicit dependence in the index $i$ in the following of the presentation.

Now, we let  $r_1,r_2>0$ be such that $2r_1<r_2$ and small enough such that the \emph{closed} balls $B_-(0,r_1)\subset H$ and $B_+(0,r_2)\subset H$ such that
\begin{align*}
B_-(0,2r_1)+B_+(0,r_2)\subset \varphi(U_{\epsilon_i}(x_i)).
\end{align*}
Now, we define for all $0<s\leq 2r_1$ and $0<t\leq r_2$
\begin{align*}
	C(s,t)=\varphi^{-1}(B_-(0,s)+B_+(0,t))\subset U_{\epsilon_i}(x_i)\subset X.
\end{align*}
Now, fix $0<\delta<r_2^2-4r_1^2$, and let $\zeta:\R\rightarrow [0,1]$ be a smooth cut-off function such that $\mathrm{supp}(\zeta)\subset \R_+$ and $\zeta(t)=1$ for all $t\geq 1$. Now, we define a map $\Phi:X\rightarrow X$ such that
\begin{align*}
	\begin{alignedat}{2}
	&\Phi(x)=x\quad&&\text{for all}\;\, x\in X\setminus C(2r_1,r_2)\\
	&\Phi(x)=\varphi^{-1}\left(\zeta\left(\frac{\norm{\varphi(x)_-}}{r_1}-1\right)\varphi(x)_++\varphi(x)_-\right)\quad&& \text{for all}\;\, x\in C(2r_1,r_2).
	\end{alignedat}
\end{align*}
\begin{lemme}\label{lazer-solimini}
	The map $\Phi:X\rightarrow X$ is continuous on $X\setminus \varphi^{-1}(B_-(0,2r_1)+ \partial B_+(0,r_2))$,
	\begin{align}\label{notin}
		X\cap\ens{x:F(x)\leq \beta_0+\delta}\subset X\setminus \varphi^{-1}(B_-(0,2r_1)+\partial B_+(0,r_2)),
	\end{align}
	 and the function $\Phi$ is Lipschitz on $X\cap\ens{x: F(x)\leq \beta_0+\delta}$. Furthermore, it admits the following properties:
	\begin{enumerate}
		\item[\emph{(1)}] For all $x\in X$, then $F(\Phi(x))\leq F(x)$.
		\item[\emph{(2)}] If $x\in \partial C(r_1,r_2)$ and $F(x)\leq F(x_i)+\delta$, then $\Phi(x)\in \varphi^{-1}(\partial B_-(0,r_1))$.
	\end{enumerate}
\end{lemme}
\begin{proof}
	To check \eqref{notin}, it suffices by taking complements to show that for all $x\in \varphi^{-1}(B_-(0,2r_1)+\partial B_+(0,r_2))$, we have
	\begin{align*}
		F(x)>\beta_0+\delta.
	\end{align*}
	For all $x\in \varphi^{-1}(B_-(0,2r_1)+\partial B_+(0,r_2))$, we have $\norm{\varphi(x)_+}=r_2$, and $\norm{\varphi(x)_-}\leq 2r_1$, so that by \eqref{morse}
	\begin{align*}
		F(x)=\beta_0+\norm{\varphi(x)_+}^2-\norm{\varphi(x)_-}^2=\beta_0+r_2^2-\norm{\varphi(x)_-}^2\geq \beta_0+r_2^2-4r_1^2>\beta_0+\delta
	\end{align*}
	by definition of $0<\delta<r_2^2-4r_1^2$, which shows the claim.
	
	$(1)$ As $\Phi=\mathrm{Id}$ on $X\setminus C(2r_1,r_2)$, it suffices to check the property on $C(2r_1,r_2)$. If $x\in C(2r_1,r_2)
	$, then by \eqref{morse} and as $\zeta\leq 1$
	\begin{align*}
		F(\Phi(x))&=F\left(\varphi^{-1}\left(\zeta\left(\frac{\norm{\varphi(x)_-}}{r_1}-1\right)\varphi(x)_++\varphi(x)_-\right)\right)\\
		&=F(x_i)+\zeta\left(\frac{\norm{\varphi(x)_-}}{r_1}-1\right)^2\norm{\varphi(x)_+}^2-\norm{\varphi(x)_-}^2\\
		&\leq F(x_i)+\norm{\varphi(x)_+}^2-\norm{\varphi(x)_-}^2=F(x).
	\end{align*}
	
	$(2)$ If $x\in \partial C(r_1,r_2)$ and $F(x)\leq F(x_i)+\delta$, recalling that $0<\delta<r_2^2-4r_1^2$, we see that
	\begin{align*}
		F(x)=F(x_i)+\norm{\varphi(x)_+}^2-\norm{\varphi(x)_-}^2<F(x_i)+r_2^2-4r_1^2.
	\end{align*}
	Therefore, as $x\in \partial C(r_1,r_2)$ we must have $\norm{\varphi(x)_-}=r_1$, so that 
	\begin{align*}
		\zeta\left(\frac{\norm{\varphi(x)_-}}{r_1}-1\right)=0
	\end{align*}
	and $\Phi(x)=\varphi^{-1}\left(\varphi(x)_-\right)$, and as $\norm{\varphi(x)_-}=r_1$, this exactly means that $\Phi(x)\in \varphi^{-1}(\partial B_-(0,r_1))$.
\end{proof}

\begin{rem}
	It is also claimed (without proof, which is left to the reader) in \cite{lazer} and \cite{ghoussoub2} that we have the additional property:
	\begin{enumerate}
	\item[\emph{\emph{(3)}}] We have $\Phi(X\setminus C(r_1,r_2))\subset X\setminus C(r_1,r_2)$.  
\end{enumerate}
	
	As $\Phi=\mathrm{Id}$ on $X\setminus C(2r_1,r_2)$, we indeed have trivially
	\begin{align*}
		\Phi(X\setminus C(2r_1,r_2))=X\setminus C(2r_1,r_2)\subset X\setminus C(2r_1,r_2).
	\end{align*}
	Therefore, the property is equivalent to 
	\begin{align*}
		\Phi\left(C(2r_1,r_2)\setminus \mathrm{int}(C(r_1,r_2))\right)\subset X\setminus \mathrm{int}(C(r_1,r_2))
	\end{align*}
	Let $x\in C(2r_1,r_2)\setminus \mathrm{int}(C(r_1,r_2))$ be a fixed element. Then at least one of the properties $r_1\leq \norm{\varphi(x)_+}\leq 2r_1$ or $\norm{\varphi(x)_+}=r_2$ holds. Furthermore, as
	\begin{align*}
		\varphi(\Phi(x))=\zeta\left(\frac{\norm{\varphi(x)_+}}{r_1}-1\right)\varphi(x)_++\varphi(x)_-,
	\end{align*}
	we trivially obtain
	\begin{align*}
		\varphi(\Phi(x))_+=\zeta\left(\frac{\norm{\varphi(x)_+}}{r_1}-1\right)\varphi(x)_+,\quad \varphi(\Phi(x))_-=\varphi(x)_-.
	\end{align*}
	Therefore, $\Phi(x)\in \mathrm{int}(C(r_1,r_2))=U_{-}(0,r_1)+U_+(0,r_2)$ if and only if
	\begin{align}\label{1}
		\norm{\varphi(\Phi(x))_+}=\zeta\left(\frac{\norm{\varphi(x)_-}}{r_1}-1\right)\norm{\varphi(x)_+}<r_2\quad \text{and}\qquad  \norm{\varphi(\Phi(x))_-}=\norm{\varphi(x)_-}<r_1.
	\end{align}
	The second inequality in \eqref{1} implies that $\norm{\varphi(x)_-}<r_1$, so $\norm{\varphi(x)_+}= r_2$ (as $x\in C(2r_1,r_2)\setminus \mathrm{int}(C(r_1,r_2))$, and \eqref{1} is equivalent to
	\begin{align*}
		\zeta\left(\frac{\norm{\varphi(x)_-}}{r_1}-1\right)<1,
	\end{align*}
	and by construction of $\zeta$, we see that there exists $0<\delta<1$ such that
	\begin{align*}
		\zeta\left(\frac{t}{r_1}-1\right)<1\;\, \text{for all}\;\, t<\delta (2r_1).
	\end{align*}
	This implies that 
	\begin{align*}
		\Phi(\varphi^{-1}(B_-(0,\delta (2r_1))+\partial B_+(0,r_2)))\subset \mathrm{int}(C(r_1,r_2))
	\end{align*}
	and as trivially 
	\begin{align}\label{2}
		\varphi^{-1}(B_-(0,\delta( 2r_1))+\partial B_+(0,r_2))\not\subset \mathrm{int}(C(r_1,r_2))=\varphi^{-1}\left(U_-(0,r_1)+U_+(0,r_2)\right),
	\end{align}
	we see that property $(3)$ is actually \emph{false} (as the set on the left-hand side of \eqref{2} is non-empty). However, it does not enter in the proof of the main  theorem in \cite{lazer}, as we shall see below.
\end{rem}

\begin{lemme}\label{hausdorff}
	Let $K$ be a closed set in a $d$-dimensional $C^1$ manifold $M^d$ and $H$ be a Hilbert space and let $f:K\rightarrow H$ be a continuous function such that $0\notin f(K)$. If $d<\dim H$, there exists a continuous extension $\bar{f}:M^d\rightarrow H\setminus\ens{0}$.
\end{lemme}
\begin{proof}
    First assume that $K$ is compact, and let $r>0$ such that $K\subset B(0,r)$. Then we obtain an extension $\bar{f}:M^d\rightarrow H$ by a theorem of Dugundji (see \cite{dugundji}) through partition of unity. Furthermore, as $M^d$ is a smooth manifold, we can take the partition of unity to be $C^1$ so that the restriction $\bar{f}|_{M^d\setminus K}:M^d\setminus K\rightarrow H$ is $C^1$. In particular, as $\bar{f}|_{M^d\setminus K}:M^d\setminus K\rightarrow H$ is locally Lipschitz, 
    \begin{align}\label{nosard}
    \mathrm{dim}_{\mathscr{H}}(\bar{f}(M^d\setminus K))\leq d<\mathrm{dim}\, H,
    \end{align}
    where $\mathrm{dim}_{\mathscr{H}}$ designs the Hausdorff measure of the metric space $H$ induced with its natural distance. In particular, as $0\notin \bar{f}(K)=f(K)
    $ by assumption, and as $\bar{f}(M^d\setminus K)$ cannot contain an open ball by \eqref{nosard} (otherwise it would be of Hausdorff dimension $\dim H\geq d+1$), we deduce that $B(0,r)\not\subset  \bar{f}(M^d)$. In particular, if $x_0\in B(0,r)\subset \bar{f}(M^d)$ is a fixed point, we can project $\bar{f}(\R^n)\cap B(0,r)$ on $\partial B(0,r)$ to obtain the required extension.
    
    If $K$ is not compact, we fix some arbitrary point $p\in M^d$ and for all $n\in \N$, we let $K_n=K\cap\bar{B}(p,n)$. We apply the previous construction to the restriction $f_{K_1}:K_1\rightarrow H\setminus\ens{0}$ to obtain an extension $\bar{f}_{K_1}:M^d\rightarrow H\setminus\ens{0}$ . Now, let $\bar{f}_1:\bar{B}(p,1)\cup K\rightarrow H\setminus\ens{0}$ be the extension by $f$ on $K\setminus \bar{B}(p,1)$ of the restriction ${\bar{f}_{K_1}}_{|\bar{B}(p,1)}:\bar{B}(p,1)\rightarrow H\setminus\ens{0}$. This gives a family of functions $f_n:\bar{B}(p,n)\cup K\rightarrow H\setminus\ens{0}$ such that for all $m\geq n$, $f_n=f_m$ on $\bar{B}(p,n)\cup K$, so all these functions have a common extension $\bar{f}:M^d\rightarrow H\setminus\ens{0}$, and this concludes the proof of the Lemma.
\end{proof}

\begin{rem}
	As, we only use the Lipschitz property, the proof would carry one to metric spaces of Hausdorff dimension at mist $d$ admitting Lipschitz partitions of unity - in particular, this would work for Lipschitz manifolds (notice that the part using Dugundji extension theorem works for any metric space). More generally, we observe that the Lemma would still hold if the function $\bar{f}|_{M^d\setminus K}:M^d$ was only $\alpha$-H\"{o}lder with $\alpha>\dfrac{d}{d+1}$, so we could relax the hypothesis to metric spaces admitting $\alpha$-H\"{o}lder partitions of unity.
\end{rem}

\subsection{The index bounds for non-degenerate functions on Finsler-Hilbert manifolds}

\begin{defi}
	If $\mathscr{A}$ is a min-max family and $F\in C^1(X,\R)$, and $\ens{A_k}_{k\in \N}\subset \mathscr{A}$ is such that
	\begin{align*}
		\sup F(A_k)\conv{k\rightarrow \infty}\beta(F,\mathscr{A})=\inf_{A\in \mathscr{A}}\sup F(A),
	\end{align*}
	we define
	\begin{align*}
		A_{\infty}=X\cap \ens{x: x=\lim\limits_{k\rightarrow \infty}x_k,\;\, \mathrm{dist}(x_k,A_k)\conv{k\rightarrow \infty}0}.
	\end{align*}
\end{defi}

\begin{theorem}[Lazer-Solimini, \cite{lazer}]\label{lz}
	Let $X$ be a $C^2$ Finsler-Hilbert manifold, $\mathscr{A}$ (resp. $\mathscr{A}^{\ast}$, resp. $\widetilde{\mathscr{A}}$) be a $d$-dimensional admissible family (resp. dual family, resp. co-dual family) with boundary $\ens{C_i}_{i\in I}\subset X$ and let $F\in C^2(X,\R)$ be such that $F$ satisfies the Palais-Smale at level $\beta_0=\beta(F,\mathscr{A})$. Assume furthermore that all critical points of $F$ are non-degenerate at level $\beta_0$, and that the min-max is non-trivial, \textit{i.e.}
	\begin{align*}
		\beta_0&=\inf_{A\in \mathscr{A}}\sup F(A)>\sup_{i\in I}\,\sup F(C_i)=\widehat{\beta}_0\\
		\Big(\text{resp. } \beta_0^{\ast}&=\inf_{A\in \mathscr{A}^{\ast}}\sup F(A)>\sup_{i\in I}\,\sup F(C_i)=\widehat{\beta}_0\Big)\\
		\Big(\text{resp. } \widetilde{\beta}_0&=\inf_{A\in \widetilde{\mathscr{A}}}\sup F(A)>\sup_{i\in I}\,\sup F(C_i)=\widehat{\beta}_0\Big).
	\end{align*}
	 Then for all $\ens{A_k}_{k\in \N}\subset \mathscr{A}$ such that
	 \begin{align*}
	 	&\sup F(A_k)\conv{k\rightarrow \infty}\beta_0,\\
	 	\Big(\text{resp.}\;\, &\sup F(A_k)\conv{k\rightarrow\infty}\beta_0^{\ast} \Big)\\
	 	\Big(\text{resp.}\;\, &\sup F(A_k)\conv{k\rightarrow \infty}\widetilde{\beta}_0\Big)
	 \end{align*}
	  the exists $x\in K(F,\beta_0)$ (resp. $x^{\ast}\in K(F,\beta_0^{\ast})$, resp. $\widetilde{x}\in K(F,\widetilde{\beta}_0)$) and a sequence $\ens{x_k}_{k\in \N}\subset X$ such that $x_k\in A_k$ for all $k\in \N$, $x_k\conv{k\rightarrow \infty}x$ (resp. $x^{\ast}$, resp. $\widetilde{x}$) and
	 \begin{align*}
	       \mathrm{Ind}_{F}(x)\leq d,\quad (\text{resp.}\;\,      \mathrm{Ind}_{F}(x^{\ast})\geq d,\;\, \text{resp.}\;\,
	       \mathrm{Ind}_{F}(\widetilde{x})=d).
	  \end{align*}
\end{theorem}
\begin{rem}
	The proof shows that it suffices to assume that $F$ is non-degenerate on $K(F,\beta_0)\cap A_{\infty}$.
\end{rem}

This is easy to see that the proof is reduced to the following Theorem (from it one obtains immediately Theorem \eqref{lz}, as we shall see shortly). 

\begin{prop}[Lazer-Solimini, \cite{lazer}, Solimini, Lemma $2.19$ \cite{solimini}]\label{ext}
	Let $F\in C^2(X,\R_+)$ as in \emph{Theorem \ref{lz}}, let $\mathscr{A}$ (resp. $\mathscr{A}^{\ast}$, or $\bar{\mathscr{A}}$) be a $d$-dimensional admissible min-max family and assume that all critical points at level $\beta_0$ (resp. at level $\beta^{\ast}_0$, or $\widetilde{\beta}_0$) are non-degenerate, and assume that $x_0\in K(F,\beta_0)$ \emph{(}resp. $x_0\in K(F,\beta^{\ast}_0)$, resp. ${x}_0\in K(F,\widetilde{\beta}_0)$\emph{)} satisfies the estimate
	\begin{align}\label{badindex1}
	\mathrm{Ind}_{F}(x_0)
	>d
	\end{align}
	respectively for $\mathscr{A}^{\ast}$
	\begin{align}
	\mathrm{Ind}_{F}(x_0)
	<d
	\end{align}
	and for $\widetilde{\mathscr{A}}$
	\begin{align}
	\mathrm{Ind}_{F}(x_0)
	\neq d.
	\end{align}
	Then for all small enough $\epsilon>0$, there exists $\delta>0$ such that for all $A\in \mathscr{A}$ \emph{(}resp. $\mathscr{A}^{\ast}$, or $\bar{\mathscr{A}}$\emph{)}, $\sup F(A)\leq \beta_0+\delta$ \emph{(}resp. $\sup F(A)\leq \beta_0^{\ast}+\delta$, resp. $\sup F(A)\leq \widetilde{\beta}_0+\delta$ \emph{)} implies that there exists $A'\in \mathscr{A}$ (resp. $\mathscr{A}^{\ast}$, or $\bar{\mathscr{A}}$) such that
	\begin{align}\label{p3}
	\left\{\begin{alignedat}{1}
	&A\setminus U_{2\epsilon}(x_0)=A'\setminus U_{2\epsilon}(x_0)\\
	&A'\cap U_{\epsilon}(x_0)=\varnothing\\
	&\sup F(A')\leq \sup F(A).
	\end{alignedat}\right.
	\end{align}
\end{prop}
\begin{proof}
	\textbf{Case 1: admissible families.}
	
	Taking the previous notations of Lemma \ref{lazer-solimini}, we will show that for $0<\delta<r_2^2-4r_1^2$, there exists $A'\in \mathscr{A}$ such that
	\begin{align}\label{refined}
	\left\{
	\begin{alignedat}{1}
	    &A'\setminus C(2r_1,r_2)=A\setminus C(2r_1,r_2)\\
		&A'\cap \mathrm{int}(C(r_1,r_2))=\varnothing \\
		&\sup F(A')\leq \sup F(A)\leq \beta_0+\delta.
		\end{alignedat}\right.
	\end{align}
	First, let $f\in C^0(M^d,X)$ such that $A=f(M^d)$, and consider the open subset $U=f^{-1}(\mathrm{int}(C(r_1,r_2)))\subset M^d$. For all $p\in \partial U=f^{-1}(\partial C(r_1,r_2))\subset M^d$, we have by definition $f(p)\in \partial C(r_1,r_2)$, and 
	\begin{align*}
		F(f(p))\leq \sup F(A)\leq \beta_0+\delta,
	\end{align*}
	so by Lemma \ref{lazer-solimini} $(2)$, we have $\Phi(f(p))\in \varphi^{-1}(\partial B_-(0,r_1))$. As $p\in \partial U$ was arbitrary we obtain
	\begin{align}\label{step1}
		\Phi(f(\partial U))\subset \varphi^{-1}(\partial B_-(0,r_1)).
	\end{align}
	Now, $\varphi:\varphi^{-1}(B_-(0,r_1))\rightarrow B_-(0,r_1)$ is a (Lipschitz) homeomorphism, so it induces a homeomorphism on the boundary
	\begin{align*}
		\varphi:\varphi^{-1}(\partial B_-(0,r_1))\rightarrow \partial B_-(0,r_1).
	\end{align*}
	Furthermore, as $\partial B_-(0,r_1)\subset H_-$ is a retract by deformation of $H_-\setminus\ens{0}$, we see that by Lemma \ref{hausdorff} that $\varphi\circ \Phi\circ f:\partial U\rightarrow \partial B_-(0,r_1)\subset H_-\setminus\ens{0}$ can be extended as a map $\Psi:\bar{U}\rightarrow \partial B_-(0,r_1)$ (by using the projection $H_-\setminus\ens{0}\rightarrow \partial B_-(0,r_1)$\,), and the map $\bar{\Phi\circ f}=\varphi^{-1}\circ \Psi:\bar{U}\rightarrow \varphi^{-1}(\partial B_-(0,r_1))$ furnishes a continuous extension of $\Phi\circ f:\partial U\rightarrow \varphi^{-1}(\partial B_-(0,r_1))$. Now, define the \emph{continuous} map $\widetilde{f}:M^d\rightarrow X$ by 
	\begin{align*}
		\widetilde{f}(p)=\left\{\begin{alignedat}{2}
		&f(p)\quad &&\text{for all}\;\, p\in M^d\setminus \bar{U},\\
		&\bar{\Phi\circ f}(p)\quad&& \text{for all}\;\, p\in \bar{U}.
		\end{alignedat}\right.
	\end{align*}
	We first need to check that $A'=\widetilde{f}(M^d)$ satisfies the non-triviality of the boundary condition. First, up to taking $r_1,r_2>0$ smaller, as 
	\begin{align*}
		F(x_0)=\beta_0>\sup_{i\in I}\sup F(h_i(B_i^{d-1}))
	\end{align*}
	we can assume that $C(2r_1,r_2)\cap h_i(B_i^{d-1})=\varnothing$ as $F$ is continuous. In particular, as $\Phi=\mathrm{Id}$ on $X\setminus C(2r_1,r_2)$, we have $\widetilde{f}|_{B_i^{d-1}}=f_{|B^{d-1}}$ on $B_i^{d-1}$ for all $i\in I$, so $A'\in \mathscr{A}$. Furthermore, for all $p\in M^d\setminus f^{-1}(C(2r_1,r_2))$, $\widetilde{f}(p)=f(p)$, so
	\begin{align*}
		F(\widetilde{f}(p))=F(f(p))\leq \sup F(A)\leq \beta_0+\delta.
	\end{align*}
	Then, for all $p\in f^{-1}(C(2r_1,r_2)\setminus \mathrm{int}(C(2r_1,r_2)))$, we have by Lemma \ref{lazer-solimini} $(1)$
	\begin{align*}
		F(\widetilde{f}(p))=F(\Phi(f(p)))\leq F(f(p))\leq \sup F(A)\leq \beta_0+\delta,
	\end{align*} 
	and finally, for all $p\in f^{-1}(\mathrm{int}(C(r_1,r_2)))=U$, we have by construction $\widetilde{f}(p)\in \varphi^{-1}(\partial B_-(0,r_1))$, but this implies by \eqref{morse} that
	\begin{align*}
		F(\widetilde{f}(p))=\beta_0-\norm{\varphi(\widetilde{f}(p))_-}^2<\beta_0\leq \sup F(A).
	\end{align*}
	Finally, as $A'\cap \mathrm{int}(C(r_1,r_2))=\varnothing$ and $A'\setminus C(2r_1,r_2)=A\setminus C(2r_1,r_2)$, this proves \eqref{refined}.
	
	\textbf{Case 2: dual admissible families.}
	
	 In this case, the construction is straightforward, as we will show that under the same notations for the Morse transformation, we have for all $0<\eta<\delta$ and for all $A\in \mathscr{A}^{\ast}$ such that
	\begin{align*}
		\sup F(A)\leq \beta^{\ast}_0+\eta,
	\end{align*}
	there holds (notice that $\Phi(A)\in \mathscr{A}^{\ast}$ by construction of $\Phi$)
	\begin{align}\label{cont0}
		A'=\Phi(A)\setminus \mathrm{int}(C(r_1,r_2))\in \mathscr{A}^{\ast},
	\end{align}
	which will immediately imply the claim, as $F(\Phi(x))\leq F(x)$ for all $x\in X$, so that
	\begin{align*}
		\sup F(A')=\sup F(\Phi(A))\leq \sup F(A)\leq \beta^{\ast}_0+\eta.
	\end{align*}
	Now assume by contradiction that \eqref{cont0} does not hold. This means by Definition \ref{defminmax} that there exists a continuous map  $h:\Phi(A)\setminus \mathrm{int}(C(r_1,r_2))\rightarrow \R^d\setminus \ens{0}$ such that for some $i\in I$ and $j\in J_i$, we have
	\begin{align*}
		h(x)=h_i^j(x)\quad \text{for all}\;\, x\in C_i\cap \Phi(A).
	\end{align*}
	Now, consider the restriction $h\circ \varphi^{-1}:\varphi(\Phi(A)\cap \partial C(r_1,r_2))\rightarrow \R^d\setminus\ens{0}$. As $\varphi(\Phi(A)\cap \partial C(r_1,r_2))\subset H_-$ and $\mathrm{dim}(H_-)=\mathrm{Ind}_{F}(x_0)<d$, we deduce by Lemma \ref{lazer-solimini} that there exists an extension 
	\[\bar{h\circ \varphi^{-1}}:\varphi(\Phi(A)\cap C(r_1,r_2)\rightarrow \R^d\setminus\ens{0},\] 
	and 
	\[\bar{h}=\bar{h\circ \varphi^{-1}}\circ \varphi:\Phi(A)\cap C(r_1,r_2)\rightarrow \R^d\setminus\ens{0}\]
	 is a continuous extension of $h|_{\Phi(A)\cap \partial C(r_1,r_2)}:\Phi(A)\cap \partial C(r_1,r_2)\rightarrow \R^d\setminus\ens{0}$. Finally, if $\widetilde{h}:\Phi(A)\rightarrow \R^d\setminus\ens{0}$ is the \emph{continuous} map given by
	 \begin{align*}
	 	\widetilde{h}(x)=\left\{\begin{alignedat}{2}
	 	&h(x),\quad &&\text{for all}\;\, x\in \Phi(A)\setminus \mathrm{int}(C(r_1,r_2)),\\
	 	&\bar{h}(x),\quad && \text{for all}\;\, x\in \Phi(A)\cap C(r_1,r_2)
	 	\end{alignedat}\right.
	 \end{align*}
	 this implies by definition of $\mathscr{A}^{\ast}$ that $\Phi(A)\notin \mathscr{A}^{\ast}$, a contradiction (as $0\notin \mathrm{Im}(\widetilde{h})$).
	
	\textbf{Case 3: co-dual admissible families.}
	
	 First, the argument of \textbf{Case 2} shows that we only need to treat the case $\mathrm{Ind}_{F}(x_0)<d$, as the map $\varphi:U_{\epsilon}(x_0)\rightarrow \varphi(U_{\epsilon}(x_0))\subset H_-$ is a locally bi-Lipschitz homeomorphism, so the map $\Phi:X\rightarrow X$ is locally Lipschitz on $A$, so that
	\begin{align}\label{haus}
	\mathrm{dim}_{\mathscr{H}}(\Phi(A))\leq \mathrm{dim}_{\mathscr{H}}(A)<d+1
	\end{align}
	and as $\Phi(A)\in \mathscr{A}^{\ast}$, we obtain by \eqref{haus} that $\Phi(A)\in \widetilde{\mathscr{A}}$. 
	
	Therefore, we see that we can assume that $\mathrm{Ind}_{F}(x_0)=\mathrm{dim}(H_-)\geq d+1$. Once again, as the map $\varphi:U_{\epsilon}(x_0)\rightarrow \varphi(U_{\epsilon}(x_0))\subset H_-$ is a locally bi-Lipschitz homeomorphism, and $\Phi: X\rightarrow X$ is locally Lipschitz on $A$, we have
	\begin{align}\label{dim}
		\mathrm{dim}_{\mathscr{H}}\left(\varphi(\Phi(A)\cap C(r_1,r_2))\right)\leq \mathrm{dim}_{\mathscr{H}}(A)<d+1.
	\end{align}
	Now, we trivially have by \eqref{dim}
	\begin{align}\label{cond}
	\mathrm{dim}(B_-(0,r_1))=\mathrm{dim}(U_-(0,r_1))=\mathrm{dim}_{\mathscr{H}}(H_-)\geq d+1>\mathrm{dim}_{\mathscr{H}}\left(\varphi(\Phi(A)\cap C(r_1,r_2))\right)
	\end{align}
	In particular, we deduce from \eqref{cond} that
	\begin{align}\label{that}
		B_-(0,r_1)\not\subset \varphi(\Phi(A)\cap C(r_1,r_2)).
	\end{align}
	Now, as $\varphi(\Phi(A)\cap C(r_1,r_2))$ is closed, there exists $\eta>0$ and $x_0\in U_-(0,r_1)$ such that
	\begin{align*}
		B(x_0,\eta)\cap H_-\subset B(0,r_1)\setminus \varphi(\Phi(A)\cap C(r_1,r_2)).
	\end{align*}
	Furthermore, as the projection $\pi: B_-(0,r_1)\setminus \ens{x_0}\rightarrow \partial B_-(0,r_1)$ is Lipschitz outside of $B(x_0,\eta)$, we see that  
	\begin{align*}
		A'=(\varphi^{-1}\circ \pi) (\varphi(\Phi(A)\cap C(r_1,r_2))\in \widetilde{\mathscr{A}}
	\end{align*}
	thanks of \eqref{dim} and as $\varphi^{-1}\circ \pi$ is locally Lipschitz on $\varphi(\Phi(A)\cap C(r_1,r_2))$. By definition, we have
	\begin{align*}
		A'\cap \mathrm{int}(C(r_1,r_2))=\varnothing.
	\end{align*}
	We finally check that
	\begin{align*}
		\sup F(A')\leq \sup F(A)\leq \widetilde{\beta}_0+\delta.
	\end{align*}
	By Lemma \ref{lazer-solimini}, we have
	\begin{align}\label{e1}
		\sup F(\Phi(A))\leq \sup F(A)
	\end{align}
	and as $A'\setminus \Phi(A)\subset \varphi^{-1}(\partial B_-(0,r_1))$, we obtain by \eqref{morse} 
	\begin{align}\label{e2}
		F(x)=\widetilde{\beta}_0-\norm{\varphi(x)_-}^2<\widetilde{\beta}_0\leq \sup F(A),\quad \text{for all}\;\, x\in A'\setminus \Phi(A),
	\end{align}
	so that by \eqref{e1} and \eqref{e2}
	\begin{align*}
		\sup F(A')\leq \sup F(A)\leq \widetilde{\beta}_0+\delta,
	\end{align*}
	which concludes the proof of the theorem.
\end{proof}
\begin{proof}(of Theorem \ref{lz})
	As the conclusions of  Proposition \ref{ext} are independent of the admissible family, we can assume that $\ens{A_k}_{k\in \N}\subset X$ is such that $A_k\in \mathscr{A}$ for all $k\in \N$ and
	\begin{align*}
		\sup F(A_k)\conv{k\rightarrow \infty}\beta_0.
	\end{align*}
	Now, let 
	\begin{align*}
		A_{\infty}=X\cap\ens{x=\lim\limits_{k\rightarrow \infty}x_k: \;\, \mathrm{dist}(x_k,A_k)\conv{k\rightarrow \infty}0}.
	\end{align*}
	Then by assumption, $F$ is non-degenerate on $K(F,\beta_0)\cap A_{\infty}$, and as $K(F,\beta_0)\cap A_{\infty}$ is compact by the Palais-Smale condition, we deduce by the Morse lemma that $K(F,\beta_0)\cap A_{\infty}$ is finite, so we have for some $x_1,\cdots,x_m\in X$
	\begin{align*}
		K(F,\beta_0)\cap A_{\infty}=\ens{x_1,\cdots,x_m}
	\end{align*}
	Now, thanks of Proposition \ref{ext}, as $K(F,\beta_0)\cap A_{\infty}$ is finite, there exists $\delta,\epsilon>0$ such that for all $A\in \mathscr{A}$, $\sup F(A)\leq \beta_0+\delta$, there exists for all $1\leq i\leq m$ an element $A_i'\in \mathscr{A}$ such that 
	\begin{align}\label{p4}
	\left\{\begin{alignedat}{1}
	&A\setminus U_{2\epsilon}(x_i)=A_i'\setminus U_{2\epsilon}(x_i)\\
	&A_i'\cap U_{\epsilon}(x_i)=\varnothing\\
	&\sup F(A_i')\leq \sup F(A).
	\end{alignedat}\right.
	\end{align}
	Now, we taking $\epsilon>0$ sufficiently small, we can assume that 
	\begin{align}\label{p5}
	\bar{B}_{2\epsilon}(x_i)\cap \bar{B}_{2\epsilon}(x_j)=\varnothing,\quad \text{for all}\;\, i\neq j\;\,\text{with}\;\, i,j\in \ens{1,\cdots,m}
	\end{align} 
	Thanks of \eqref{p4} and \eqref{p5}, we see that $A_k$ satisfies the hypothesis to obtain \eqref{p4} for $k$ large enough define by a finite induction $A_k^1,\cdots A_k^m\in \mathscr{A}$ by
	\begin{align*}
		A_k^1=(A_k)_1',\quad A_k^{i}=(A_k^{i-1})_i,\quad \text{for all}\;\, 2\leq i\leq m.
	\end{align*}
	Then $A_k^m\in \mathscr{A}$ and 
	\begin{align*}
		\sup F(A_k^m)\leq \sup F(A_{k})\conv{k\rightarrow \infty}\beta_0,
	\end{align*} 
	so by any deformation lemma (see \emph{e.g.} \cite{solimini}), there exists $\ens{x_k^m}_{k\in \N}$ such that $x_k^m\in A_k^m$ for all $k\in \N$, and
	\begin{align*}
		\mathrm{dist}(x_k^m,\ens{x_1,\cdots,x_m})\geq \epsilon,\quad \text{for all}\;\, k\in \N.
	\end{align*}
	and $x_k\conv{k\rightarrow \infty} x_{\infty}\in K(F,\beta_0)$. Furthermore, assuming that $\epsilon>0$ is small enough, and as $\ens{x_1,\cdots,x_m}=K(F,\beta_0)\cap A_{\infty}$ are isolated, we can assume that $K(F,\beta_0)\cap A_{\infty}$ is isolated in $K(F,\beta_0)$, so that
	\begin{align*}
		\mathrm{dist}(x_k^m,K(F,\beta_0))\geq \epsilon\;\, \text{for all}\;\, k\in \N,
	\end{align*}
    which furnishes the desired contradiction.
\end{proof}

\begin{prop}\label{vietoris}
Let $d\geq 1$ be a fixed integer, $R$ be an arbitrary ring, $G$ be an abelian group, $F\in C^2(X,\R_+)$ as in \emph{Theorem \ref{lz}}, $B\subset X$ a compact subset, $\alpha_{\ast}\in H_d(X,B,R)\setminus\ens{0}$ and $\alpha^{\ast}\in H^d(X,G)\setminus\ens{0}$ be non-trivial classes in relative homology and cohomology respectively, let  $\mathscr{A}(\alpha_{\ast})$ and $\mathscr{A}^{\ast}$ be the corresponding $d$-dimensional homological and cohomological admissible families, and 
\begin{align*}
	\beta_0=\beta(F,\mathscr{A}(\alpha_{\ast}))=\inf_{A\in \mathscr{A}(\alpha_{\ast})}\sup F(A),\qquad \bar{\beta}_0=\beta(F,\mathscr{A}(\alpha^{\ast}))=\inf_{A\in \mathscr{A}}\sup F(A)
\end{align*}
be the associated width.
 Assume that $x_0\in K(F,\beta_0)$ (resp. $x_0\in K(F,\bar{\beta}_0)$) is a \emph{non-degenerate} critical points of $F$ at level $\beta_0$ (resp. $\bar{\beta}_0$) and that
 \begin{align}\label{neq}
 	\mathrm{Ind}_{F}(x_0)\neq d.
 \end{align}
Then for all small enough $\epsilon>0$, there exists $\delta>0$ such that for all $A\in \mathscr{A}(\alpha_{\ast})$ \emph{(}resp. $\mathscr{A}(\alpha^{\ast})$\emph{)}, $\sup F(A)\leq \beta_0+\delta$ \emph{(}resp. $\sup F(A)\leq \bar{\beta}_0+\delta$\emph{)} implies that there exists $A'\in \mathscr{A}(\alpha_{\ast})$ \emph{(}resp. $\mathscr{A}(\alpha^{\ast})$\emph{)}  such that
\begin{align}\label{p3}
\left\{\begin{alignedat}{1}
&A\setminus U_{2\epsilon}(x_0)=A'\setminus U_{2\epsilon}(x_0)\\
&A'\cap U_{\epsilon}(x_0)=\varnothing\\
&\sup F(A')\leq \sup F(A).
\end{alignedat}\right.
\end{align}
\end{prop}
\begin{proof}
	Let $x_0\in K(F,\beta_0)$ be a non-degenerate critical critical point, and let $r_1,r_2,\delta>0$ be given by Lemma \ref{lazer-solimini} such that $0<\delta<r_2^2-4r_1^2$, and $A\in \mathscr{A}(\alpha_{\ast})$ such that
	\begin{align*}
		\sup F(A)\leq \beta_0+\delta.
	\end{align*}
	Then by definition, $\alpha\in \mathrm{Im}(\iota_{A,\ast})$, where $\iota_{A,\ast}:H_{d}(A,B)\rightarrow H_d(X,B)$ is the induced map in relative homology from the inclusion $\iota_{A}:A\rightarrow X$. We will now show that for all $1/2<\epsilon<1$ close enough to $1$, we have
	\begin{align*}
		A\setminus \mathrm{int}(C(\epsilon\, r_1,\epsilon\, r_2))\cup \left( C(r_1,0)\setminus \mathrm{int}(C(\epsilon r_1,0))\right)\in \mathscr{A}(\alpha_{\ast}).
	\end{align*}
	We choose $r_1,r_2>0$ small enough such that $C(s,t)$ is closed for all $s\leq 2r_1$ and $t\leq r_2$ by Theorem \ref{palais2}. 
	 Let 
	\begin{align*}
		Y=A\cup C(r_1,0)
	\end{align*}
	and observe that
	\begin{align*}
		\mathrm{int}(Y\setminus \mathrm{int}(C(\epsilon\,r_1,\epsilon\,r_2)))\cup \mathrm{int}\left(Y\cap C(r_1,r_2)\right)=Y,
	\end{align*}
	and define for convenience of notations 
	\begin{align*}
		A_{\epsilon}(r_1,r_2)=C(r_1,r_2)\setminus \mathrm{int}(C(\epsilon\,r_1,\epsilon\,r_2)).
	\end{align*}
	Therefore, we obtain the following Mayer-Vietoris commutative diagram
\begin{center}
	\begin{tikzcd}
		H_d(Y\setminus \mathrm{int}(C(\epsilon\,r_1,\epsilon\,r_2)))\oplus H_d(Y\cap C(r_1,r_2))\rar\dar[swap, "\Phi_{\ast}"] & 
		H_d(Y)\rar\dar[swap, "\Phi_{\ast}"] & H_{d-1}(Y\cap  A_{\epsilon}(r_1,r_2))\dar\dar[swap, "\Phi_{\ast}"]\\
		H_d(\Phi(Y\setminus \mathrm{int}(C(\epsilon\,r_1,\epsilon\,r_2)))\oplus H_d(\Phi(Y\cap C(r_1,r_2)))\rar & H_d(\Phi(Y))\rar & H_{d-1}(\Phi(Y\cap  A_{\epsilon}(r_1,r_2)))
	\end{tikzcd}
\end{center}
Now, we have by the proof of Lemma \ref{lazer-solimini} that
\begin{align}\label{identities}
	&\Phi(C(r_1,r_2))=\Phi(C(r_1,0))=B_-(0,r_1)\simeq B^n(0,1)\subset \R^n,\nonumber\\ 
	&\Phi(X\cap \ens{x:F(x)\leq \beta_0+\delta}\cap \partial C(r_1,r_2))=\Phi(\partial C(r_1,0))=\varphi^{-1}(\partial B_-(0,r_1))
\end{align}
where $\simeq$ designs the equivalence up to homeomorphism, $\varphi$ is the Lipschitz local homeomorphism given by the Morse lemma, and $B_-(0,r_1)$ is the closed ball of radius $r_1$ in the Hilbert space $T_{x_0}X$ corresponding to negative space of $\D^2F(x)\in \mathscr{L}(T_xX)$. 
Let us show that for $0<\epsilon<1$ large enough, we have
\begin{align*}
	\Phi(Y\cap A_{\epsilon}(r_1,r_2))\subset \varphi^{-1}(B_-(0,r_1)\setminus \ens{0})\simeq S^{n-1}.
\end{align*}
First, let us show that for $0<\epsilon<1$ large enough, we have
\begin{align*}
	\Phi(X\cap\ens{x: F(x)\leq \beta_0+\delta}\cap A_{\epsilon}(r_1,r_2))\subset \varphi^{-1}(B_-(0,r_1)\setminus U_-(0,\epsilon r_1)).
\end{align*}
By contradiction, if there exists $x\in X\cap\ens{x: F(x)\leq \beta_0+\delta}\cap A_{\epsilon}(r_1,r_2)$ such that $\Phi(x)\in \varphi^{-1}(U_-(0,\epsilon r_1))$, as 
\begin{align}\label{form}
	\Phi(x)=\varphi^{-1}\left(\varphi(x)_-\right)
\end{align}
we must have $\norm{\varphi(x)_-}<\epsilon r_1$, and as $A_{\epsilon}(r_1,r_2)=C(r_1,r_2)\setminus \mathrm{int}(C(r_1,r_2))$, this implies that $\norm{\varphi(x)_+}\geq \epsilon r_2$, so that
\begin{align*}
      F(x)=\beta_0+\norm{\varphi(x)_+}^2-\norm{\varphi(x)_-}^2=\beta_0+\norm{\varphi(x)_+}^2\geq \beta_0+\epsilon^2r_2^2-\epsilon^2r_1^2,
\end{align*}
and as $0<\delta<r_2^2-4r_1^2$, we obtain
\begin{align*}
	\beta_0+\epsilon^2r_2^2-\epsilon^2r_1^2\leq F(x)\leq \beta_0+\delta<\beta_0+r_2^2-4r_1^2
\end{align*}
and as $0<2r_1<r_2$, this yields to a contradiction if 
\begin{align*}
	0<\sqrt{1-\frac{3r_1^2}{r_2^2-r_1^2}}\leq \epsilon<1.
\end{align*}
Furthermore, as we trivially have (by \eqref{form}, valid for all $x\in C(r_1,r_2)$)
\begin{align*}
	\Phi(C(r_1,0)\setminus \mathrm{int}(C(r_1,0)))=\Phi(\varphi^{-1}(B_-(0,r_1)\setminus U_-(0,\epsilon r_1)))=\varphi^{-1}(B_-(0,r_1)\setminus U_-(0,\epsilon r_1)),
\end{align*}
we obtain as $\partial B_-(0,r_1)$ is a retract by deformation of $B_-(0,r_1)\setminus U_-(0,r_1)$, we obtain the identity
\begin{align}\label{id2}
	\Phi(Y\cap A_{\epsilon}(r_1,r_2))=\varphi^{-1}(B_-(0,r_1)\setminus U_-(0,\epsilon r_1))\simeq S^{n-1}.
\end{align}
We also notice that the first equality in \eqref{identities} implies that 
\begin{align}\label{id3}
	\Phi(Y\cap C(r_1,r_2))=\varphi^{-1}(B_-(0,r_1))\simeq B^n(0,1)\subset \R^n.
\end{align}
Indeed, we have 
\begin{align*}
	Y\cap C(r_1,r_2)=(A\cap C(r_1,r_2))\cup C(r_1,0)
\end{align*}
and by \eqref{identities}, $\Phi(A\cap C(r_1,r_2))\subset \Phi(C(r_1,r_2))=\varphi^{-1}(B_-(0,r_1))$ and $\Phi(C(r_1,0))=\varphi^{-1}(B_-(0,r_1))$, which yields \eqref{id3}.

By \eqref{id2} and \eqref{id3}, we obtain
\begin{align*}
	H_d(\Phi(Y\cap C(r_1,r_2)))=\ens{0},\quad H_{d-1}(\Phi(Y\cap A_{\epsilon}(r_1,r_2)))= \ens{0}
\end{align*}
and we obtain the following exact sequence
\begin{center}
	\begin{tikzcd}
		H_d(\Phi(Y\setminus \mathrm{int}(C(\epsilon\,r_1,\epsilon\,r_2)))\arrow[r,"f"] & H_d(\Phi(Y))\arrow[r,"g"] & 0
	\end{tikzcd}
\end{center}
and as $\mathrm{Im}(f)=\mathrm{Ker}(g)=H_d(\Phi(Y))$, we deduce that $f$ is surjective. Now, as the map $\Phi:X\rightarrow X$ given by Lemma \ref{lazer-solimini} is continuous on $X\cap\ens{x:F(x)\leq \beta_0+\delta}$ and isotopic to the identity on $X\cap\ens{x:F(x)\leq \beta_0+\delta}$ (which contains  $Y$), we deduce that the $\Phi_{\ast}$ homomorphisms in the Mayer-Vietoris commutative diagram are isomorphism, so we have a surjection
\begin{align*}
	H_d(Y\setminus \mathrm{int}(C(\epsilon\,r_1,\epsilon\,r_2))) \xrightarrowdbl[]{\hphantom{100}} H_d(Y)
\end{align*}
In particular, the arrow $\bar{h}$ of the following we obtain a surjection
\begin{align}\label{arrow}
	H_d(Y\setminus \mathrm{int}(C(r_1,r_2)),B)\xrightarrowdbl[]{\hphantom{100}} H_d(Y,B).
\end{align}
Now, if $B\subset A_1\subset A_2\subset X$ are any two subsets containing $B$, we write $\iota_{A_1,A_2}:A_1\hookrightarrow A_2$ the injection and $\iota_{A_1,A_2\ast}:H_d(A_1,B)\rightarrow H_d(A_2,B)$ the induced map in homology. As $A\subset A\cup C(r_1,0)=Y\subset X$, and $\iota_{A,X}=\iota_{Y,X}\circ \iota_{A,Y}$ we have
\begin{align*}
	\iota_{A,X\ast}=\iota_{Y,X\ast}\circ \iota_{A,Y\ast},
\end{align*}
and as $\alpha_{\ast}\in \mathrm{Im}(\iota_{A,X,\ast})\subset H_d(X,B)$, this implies that $\alpha_{\ast}\in \mathrm{Im}(\iota_{Y,X\ast})$, and by the surjectivity of the arrow in \eqref{arrow}, we obtain
\begin{align*}
	\alpha_{\ast}\in \mathrm{Im}\left(\iota_{Y\setminus \mathrm{int}(C(\epsilon\,r_1,\epsilon\,r_2)),X\ast}\right),
\end{align*}
which by definition means that (notice that $Y$ is compact)
\begin{align*}
	Y\setminus \mathrm{int}(C(\epsilon\,r_1,\epsilon\,r_2))\in \mathscr{A}(\alpha_{\ast}).
\end{align*}
Finally, for all $x\in C(r_1,0)\setminus \mathrm{int}(C(\epsilon r_1,0))$, we have
\begin{align*}
	F(x)=\beta_0-\norm{\varphi(x)_-}\ leq \beta_0-\epsilon^2r_1^2<\beta_0\leq \sup F(A),
\end{align*}
so that
\begin{align*}
	\sup F(Y\setminus \mathrm{int}(C(\epsilon\, r_1,\epsilon\,r_2)))\leq \sup F(A).
\end{align*}
Using the exact same arguments of proof (with $A'=A\setminus \mathrm{int}(C(\epsilon\,r_1,\epsilon\,r_2))\cup  C(r_1,0)\setminus \mathrm{int}(C(\epsilon\,r_1,0))$) thanks of the Mayer-Vietoris sequence for singular cohomology, we show the injectivity of the following arrow 
\begin{align*}
	H^d(A\cup C(r_1,0),G)\hookrightarrow H^d(A\setminus \mathrm{int}(C(\epsilon\,r_1,\epsilon\,r_2))\cup (C(r_1,0)\setminus \mathrm{int}(C(\epsilon r_1,0)),G)
\end{align*}
and this finishes the proof of the theorem.
\end{proof}

\begin{rem}
	We see that there is absolutely no restriction in the coefficients in (singular) homology of cohomology, as we only used Mayer-Vietoris exact sequence.
\end{rem}

\begin{cor}
	Under the hypothesis of Proposition \ref{vietoris}, if $F\in C^2(X,\R)$ and $\ens{A_k}_{k\in \N}\subset \mathscr{A}(\alpha_{\ast})$ (resp. $\ens{A_k}_{k\in \N}\subset \mathscr{A}(\alpha^{\ast})$) such that
	\begin{align*}
		\sup F(A_k)\conv{k\rightarrow \infty} \beta(F,\mathscr{A}(\alpha_{\ast})),\quad \left(\text{resp.}\;\, \sup F(A_k)\conv{k\rightarrow \infty} \beta(F,\mathscr{A}(\alpha^{\ast})) \right).
	\end{align*}
	If $K(F,\beta(F,\mathscr{A}(\alpha_{\ast})))$ contains only non-degenerate critical points, there exists a sequence $\ens{x_k}_{k\in \N}\subset X$ such that $x_k\in A_k$ for all $k\in \N$ and $x_k\conv{k\rightarrow \infty}x\in K(F,\beta(F,\mathscr{A}(\alpha_{\ast})))\cap A_{\infty}$ (resp. $\bar{x}\in K(F,\mathscr{A}(\alpha^{\ast})\cap A_{\infty}$) such that
	\begin{align}
		\mathrm{Ind}_{F}(x)=d,\quad (\text{resp.}\;\, \mathrm{Ind}_{F}(\bar{x})=d).
	\end{align}
\end{cor}
\begin{proof}
	It is exactly the same as the proof of Theorem \ref{lz}, using Proposition \ref{vietoris} instead of Proposition \ref{ext}.
\end{proof}

\subsection{Application to the min-max hierarchies for minimal surfaces}\label{tristan}

We observe that the previously considered admissible families need not be continuous with respect to the strong topology on $X$, as the following corollary shows. This application is of interest in the setting of min-max hierarchies for minimal surfaces recently  developed by Tristan Rivi\`{e}re (\cite{hierarchies}). We first introduce some terminology (see \cite{federer} chapter $4$, \cite{pitts} chapter $2$, \cite{allard} section $3$).

Let $\Sigma$ be a closed Riemann surface, $N^n$ be a compact Riemannian manifold with boundary (possibly empty) which we suppose isometrically embedded in some Euclidean space, and $\mathscr{G}_2(TN^n)$ be the Grassmannian bundle of \emph{oriented} $2$-planes in $TN^n$. We denote by $\mathscr{V}_2(N^n)$ the space of \emph{$2$-dimensional varifolds} on $N^n$, that is the space of Radon measure on $\mathscr{G}_2(TN^n)$ endowed with the weak-$\ast$ topology. Furthermore, we denote by $\mathscr{Z}_2(N^n,G)$ the space of rectifiable $2$-cycles in $N^n$ 
 with $G$-coefficients (see \cite{federer}, $4.1.24$, $4.2.26$, $4.4.1$), where $G=\Z$ or $G=\Z_2$ (or more generally, $G$ is an admissible in Almgren's sense \cite{almgrenvarifold}). It is known that every current $T\in \mathscr{Z}_k(N^n,G)$ induces a varifold $|T|\in \mathscr{V}_2(N^n)$, and we denote by $\mathcal{F}$ the flat norm on $\mathscr{Z}_2(N^n,G)$ and by $d_{\mathscr{V}}$ the varifold distance, defined for all $V,W\in \mathscr{V}_2(N^n)$ by
 \begin{align*}
 	d_{\mathscr{V}}(V,W)=\sup\ens{V(f)-W(f):\, f\in C^0_c(\mathscr{G}_2(TN^n)),\;\, \norm{f}_{\mathrm{L}^{\infty}}\leq 1,\;\, \mathrm{Lip}(f)\leq 1}.
 \end{align*}
 Furthermore, if $\phi\in \mathrm{Imm}_{3,2}(\Sigma,N^n)$ is a $W^{3,2}$ immersion as defined in Section \ref{appli}, then obviously the push-forward $\phi_{\ast}[\Sigma]$ of the current of integration $[\Sigma]$ on the closed Riemann surface $\Sigma$ is an element of $\mathscr{Z}_2(N^n,\Z)$, and furthermore, the induced varifold is denoted by $V_{\phi}=|\phi_{\ast}[\Sigma]|\in \mathscr{V}_2(N^n)$. We have explicitly for all $f\in C^0_{c}(\mathscr{G}_2(TN^n))$
 \begin{align*}
 	V_{\phi}(f)=\int_{\Sigma}f\left(\Phi(p), \phi_{\ast}T_p\Sigma\right)d\vg(p).
 \end{align*}

We introduce the following distance on $\mathscr{V}_2(N^n)\cap\ens{|T|: T\in \mathbb{Z}_2(N^n,G)}$: for all $V,W\in \mathscr{V}_2(N^n)$ such that $V=|S|$ and $W=|T|$ for some $S,T\in \mathscr{Z}_2(N^n,G)$, 
	\begin{align*}
		\mathbf{F}(S,T)=d_{\mathscr{V}}(|S|,|T|)+\mathcal{F}(S,T).
	\end{align*}
	Finally, if for all $g\in \N$, $\Sigma_g$ is a fixed closed oriented surface of genus $g$, we denote by $\mathrm{Imm}_{3,2}^0(\Sigma_g,N^n)$ the connected component (for regular homotopy) of the immersions regularly homotopic to an embedding $\Sigma_g\hookrightarrow N^n$, on we denote by $\mathrm{Imm}^{\leq g_0}(N^n)$ the disjoint union of Finsler-Hilbert manifolds
	\begin{align*}
		\mathrm{Imm}^{\leq g_0}(N^n)=\bigsqcup_{g=0}^{g_0}\mathrm{Imm}_{3,2}^0(\Sigma_g,N^n),
	\end{align*}
	We introduce for all $0\leq \sigma\leq 1$ the function $A_{\sigma}:\mathrm{Imm}^{\leq g_0}(N^n)\rightarrow \R$ defined for all $\phi\in \mathrm{Imm}^{\leq g_0}(N^n)$ by
	\begin{align*}
		A_{\sigma}(\phi)=\mathrm{Area}(\Phi(\Sigma))+\sigma^2\int_{\Sigma}\left(1+|\vec{\I}_{\phi}|^2\right)^2d\vg
	\end{align*}
	if $\phi$ is defined from a closed surface $\Sigma$, and $\vec{\I}_{\phi}$ is its second fundamental form. That $A_{\sigma}$ satisfies all hypothesis of Theorem \ref{indexbound} is verified in \cite{viscosity}.
\begin{cor}
	Let $N^n$ be a closed Riemannian manifold, $I$ be a non-empty set and let $\ens{M_i^d}_{i\in I}$ a family of $d$-dimensional cellular-complexes, for all $i\in I$, let $h_i:\partial M_{i}^d\rightarrow \mathrm{Imm}^{\leq g_0}(N^n)$ by a $\mathbf{F}$-Lipschitz map, and define
	\begin{align*}
		\mathscr{A}=\mathrm{Imm}^{\leq g_0}(N^n)\cap \ens{\phi(Y): \phi\in \mathrm{Lip}_{\mathbf{F}}(M_i^d,\mathrm{Imm}^{\leq g_0}(N^n))\;\, \text{for some}\;\, i\in I},
	\end{align*}
	and define for all $0\leq \sigma\leq 1$
	\begin{align*}
	    \beta(\sigma)=\beta(A_{\sigma},\mathscr{A})=\inf_{A\in\mathscr{A}}\sup A_{\sigma}(A)<\infty.
	\end{align*}
	Assuming that $\mathscr{A}$ is non-trivial as in Theorem \ref{indexbound}, there exists a sequence $\ens{\sigma_k}_{k\in \N}\subset (0,\infty)$ such that $\sigma_k\rightarrow 0$ and and for all $k\in \N$, there exists a critical point $x_k\in K(A_{\sigma_k})\cap \mathscr{E}(\sigma_k)$ such that
	\begin{align*}
		A_{\sigma_k}(x_k)=\beta(\sigma_k),\quad \sigma_k^2\int_{\Sigma_{\phi_k}}\left(1+|\vec{\I}_{\phi_k}|^2\right)^2d\mathrm{vol}_{g_{\phi_k}}\leq \frac{1}{\log\left(\frac{1}{\sigma_k}\right)\log\log\left(\frac{1}{\sigma_k}\right)}, \quad \mathrm{Ind}_{A_{\sigma_k}}(\phi_k)\leq d.
	\end{align*}
\end{cor}
\begin{proof}
	As the extensions are made for maps whose domains and co-domains is finite-dimensional, by the equivalence of norms in finite dimension, the different restriction of the sweep-outs are continuous in any topology, and the extension can be taken Lipschitz in the strong topology on $W^{3,2}$ immersions, so the proof is virtually unchanged.
\end{proof}

\section{Proof of the main theorem}

\subsection{The entropy condition}

Let $X$ be a Finsler manifold and $\ens{F_{\sigma}}_{\sigma\in [0,1]}\subset C^1(X,\R)$ such that for all $x\in X$, $\sigma\mapsto F_{\sigma}(x)$ is increasing. If $\mathscr{A}$ is any of the admissible families, we define for all $\sigma\in [0,1]$
\begin{align}\label{entropie}
\beta(\sigma)=\inf_{A\in\mathscr{A}}\sup F_\sigma(A)<\infty.
\end{align}
As the function $\sigma\rightarrow\beta(\sigma)$ is increasing, it is differentiable almost everywhere (with respect to the $1$-dimensional Lebesgue measure) and we have
\begin{align*}
\liminf_{\sigma\rightarrow 0}\beta'(\sigma)\left(\sigma\log\left(\frac{1}{\sigma}\right)\log\log\left(\frac{1}{\sigma}\right)\right)=0.
\end{align*}
Suppose by contradiction that this is not the case. Then there exists $\delta>0$ such that for $\sigma>0$ small enough
\begin{align*}
\beta(\sigma)-\beta(0)\geq \int_{0}^{\sigma}\beta'(t)dt\geq  \delta\int_{0}^{\sigma}\frac{dt}{t\log\left(\frac{1}{t}\right)\log\log\left(\frac{1}{t}\right)}=\infty,
\end{align*}
which contradicts \eqref{entropie}.
\begin{defi}
	
	We say that $\beta$ satisfies the entropy condition at $\sigma>0$ if $\beta$ is differentiable at $\sigma$ and if
	\begin{align*}
	\beta'(\sigma)\leq \frac{1}{\sigma\log\left(\frac{1}{\sigma}\right)\log\log\left(\frac{1}{\sigma}\right)}.
	\end{align*}
\end{defi}
In particular, there always exists a sequence of positive number $\ens{\sigma_k}_{k\in\N}$ such that $\sigma_k\conv{k\rightarrow\infty}0$ and $\beta$ verifies the entropy condition at $\sigma_k$.

\subsection{The non-degenerate case}

If $X$ is a Finsler-Hilbert manifold and $F:X\rightarrow \R$ is a $C^2$ map, we let $\D F(x)\in T_xX$ and $\D^2F(x)\in \mathscr{L}(T_xX)$ such that for all $x\in T_xX$, there holds
\begin{align*}
	&DF(x)\cdot v=\s{\D F(x)}{v}_x\\
	&D^2F(x)(v,w)=\s{\D^2F(x)v}{w}
\end{align*}

The next result is a variant of \cite{solimini}, $2.13$ \cite{ghoussoub2}, $4.5$, which will allow us to construct critical points of the right index. It permits to show that we can always obtain the entropy condition as we locate critical points in some almost critical sequence.

\begin{theorem}\label{real}
	Let $X$ be a Banach manifold and $F,G\in C^2(X,\R_+)$, $\mathscr{A}$ an admissible min-max family, and define for $0\leq \sigma<1$ the function $F_{\sigma}=F+\sigma^2G$, and
	\begin{align*}
		\beta(\sigma)=\inf_{A\in \mathscr{A}}\sup F_{\sigma}(A)<\infty,
	\end{align*}
	and assume that the \emph{\textbf{Energy bound}} \emph{$(2)$} of Theorem \ref{indexbound} holds. 
	Now suppose that $\beta$ is differentiable at $0<\sigma<1$ and satisfies the entropy condition, \emph{i.e.}
	\begin{align*}
		\beta'(\sigma)\leq \frac{1}{\log(\frac{1}{\sigma})\log\log(\frac{1}{\sigma})}.
	\end{align*}
	Let $\ens{\sigma_k}_{k\in \N}\subset(\sigma,\infty)$ be such that $\sigma_k\rightarrow \sigma$, and $\ens{A_k}_{k\in \N}\subset \mathscr{A}$ such that
	\begin{align*}
		\lim\limits_{k\rightarrow \infty}F_{\sigma_k}(A_k)\leq \beta(\sigma_k)+(\sigma_k-\sigma).
	\end{align*}
	Then for $0<\sigma\leq e^{-\frac{4}{\beta(0)}}$, there exists a sequence $\ens{x_k}_{k\in \N}\subset X$ such that for all large enough $k\in \N$
	\begin{align*}
		(1)\;\,&\mathrm{dist}(x_k,A_k)\conv{k\rightarrow \infty}0\\
		(2)\;\,&\beta(\sigma)-(\sigma_k-\sigma)\leq F_{\sigma}(x_k)\leq F_{\sigma_k}(x_k)\leq \beta(\sigma_k)+(\sigma_k-\sigma)\\
		(3)\;\,&\norm{DF_{\sigma}(x_k)}\conv{k\rightarrow \infty}0\\
		(4)\;\,&\inf_{k\in \N}F(x_k)>0.
	\end{align*}
	In particular, if $F_{\sigma}$ verifies the Palais-Smale condition at $\beta(\sigma)$, there exists $x_{\sigma}\in K_{\beta(\sigma)}\cap A_{\infty}$ such that
	\begin{align*}
		\sigma^2G(x_{\sigma})\leq \frac{1}{\log(\frac{1}{\sigma})\log\log(\frac{1}{\sigma})}.
	\end{align*}
\end{theorem}
\begin{proof}
	Looking at Step $2$ of the proof of Proposition $6.3$ (it is written for geodesics, but the same proof work equally well in general, see \cite{prop6.3}), we see that assuming by contradiction that for all  for $k\geq 1$ large enough, we have for all $x\in X$ such that $\mathrm{dist}(x,A_k)\leq \delta_k$ and
	\begin{align}\label{ent}
		\beta(\sigma)-(\sigma_k-\sigma)\leq F_{\sigma}(x)\leq F_{\sigma_k}(x)\leq \beta(\sigma_k)+(\sigma_k-\sigma),
	\end{align}
	then
	\begin{align*}
		\Vert DF_{\sigma_k}(x)\Vert \geq \delta_k>0
	\end{align*}
	for some $\delta_k>0$ to be determined later, there exists a semi-flow $\ens{\varphi^t_k}_{t\geq 0}:X\rightarrow X$ isotopic to the identity and preserving the boundary of $\mathscr{A}$ such that for all $0\leq t\leq \delta_k$ (as $\mathrm{dist}(x,\varphi^t(x))\leq t$ for all $t\geq 0$), and $x\in A_k$ such that \eqref{ent} is satisfied, there holds
	\begin{align}\label{dec}
		\frac{d}{dt}F_{\sigma}(\varphi^t_k(x))\leq -\delta_k.
	\end{align}
	In particular, as $\varphi^t_k(A)\in \mathscr{A}$, we have
	\begin{align*}
		\beta(\sigma)\leq F_{\sigma}(\varphi_k^t(A_k)),
	\end{align*}
	so we deduce that for all $0\leq t\leq \delta_k$  by \eqref{dec}
	\begin{align}\label{b1}
		\beta(\sigma)\leq \sup F_{\sigma}(\varphi^t(A_k))\leq \sup F_{\sigma}(A)-t\delta_k\leq \sup F_{\sigma_k}(A_k)-t\delta_k\leq \beta(\sigma_k)+(\sigma_k-\sigma)-t\delta_k.
	\end{align}
	Furthermore, as $\beta$ is differentiable at $\sigma$, we can assume that $k$ is large enough such that
	\begin{align}\label{b2}
		\beta(\sigma_k)\leq \beta(\sigma)+(\beta'(\sigma)+1)(\sigma_k-\sigma)
	\end{align}
	so by \eqref{b1} and \eqref{b2}, we have for $t=\delta_k$ and $\eta_k=\varphi_k^{t}:X\rightarrow X$
	\begin{align*}
		\sup F_{\sigma}(\eta_k(A))\leq \beta(\sigma)+(\beta'(\sigma)+2)(\sigma_k-\sigma)-\delta_k^2.
	\end{align*}
	Therefore, choosing 
	\begin{align*}
		\delta_k=\sqrt{2(\beta'(\sigma)+2)(\sigma_k-\sigma)},
	\end{align*}
	we find that $\eta_k(A_k)\in \mathscr{A}$ so (recall that $\beta'\geq 0$)
	\begin{align*}
		\beta(\sigma)=\inf_{A\in \mathscr{A}}\sup F_{\sigma}(A)\leq \sup F_{\sigma}(\eta_k(A_k))\leq \beta(\sigma)-2(\sigma_k-\sigma)<\beta(\sigma),
	\end{align*}
	a contradiction. Therefore, we see that there exists $x_k\in X$ such that
	\begin{align}\label{conditions}
		&(1)\;\,\mathrm{dist}(x_k,A_k)\leq \delta_k=\sqrt{2(\beta'(\sigma)+2)(\sigma_k-\sigma)}\conv{k\rightarrow \infty}0\nonumber\\
		&(2)\;\,\beta(\sigma)-(\sigma_k-\sigma)\leq F_{\sigma}(x_k)\leq F_{\sigma_k}(x_k)\leq \beta(\sigma_k)+(\sigma_k-\sigma)\nonumber\\
		&(3)'\;\,\norm{DF_{\sigma_k}(x_k)}\leq \delta_k\nonumber\\
		&(4)\;\,F(x_k)\geq \frac{3}{4}\beta(0)
	\end{align}
	where the last condition is given by the identity below $(6.11)$ in \cite{geodesics}. Finally, this is easy to see that $(3)'$ implies the $(3)$ of the theorem (thanks of the \textbf{Energy bound} condition), and this concludes the proof (see \cite{viscosity}  for the optimal hypothesis on $F_{\sigma}$ for this assertion to hold true).
\end{proof}

\begin{rem}
	If $Y\hookrightarrow X$ is a locally Lipschitz embedded Hilbert-Finsler manifold, and $\mathscr{A}\subset \mathscr{P}(Y)$ is an admissible family (\text{i.e.} it is stable under locally Lipschitz homeomorphisms of $Y$), then the restriction $F|Y$ is still $C^2$ and by taking pseudo-gradients with respect to this restriction, we see that any $A\in \mathscr{A}$ will be preserved by the map $\varphi^{t}_{\delta_k}$. Therefore, we obtain a sequence $\ens{x_k}_{k\in \N}\subset Y$ such that \eqref{conditions} are satisfied \emph{with respect to the Finsler norm and distance on $Y$}, and by the local Lipschitz embedding, we also obtain $\mathrm{dist}_X(x_k,A_k)\conv{k\rightarrow \infty}0$, and $\norm{D F_{\sigma_k}}_{X,x_k}\conv{k\rightarrow \infty}0$. Using the Palais-Smale condition and the energy bound valid with respect to $X$, the end of the proof is identical. 
\end{rem}

\begin{defi} 
	Under the previous notations, we define the set of points satisfying the entropy condition as
	\begin{align*}
		\mathscr{E}(\sigma)=X\cap\ens{x:\sigma^2G(x)\leq \frac{1}{\log(\frac{1}{\sigma})\log\log(\frac{1}{\sigma})}}.
	\end{align*}
\end{defi}

\begin{theorem}\label{ndeg}
	Let $X$ be a $C^2$ Finsler manifold, $F,G\in C^2(X,\R_+)$, and define for all $\sigma\geq 0$ the function $F_{\sigma}=F+\sigma^2G\in C^2(X,\R_+)$ and let $\mathscr{A}$ (resp. $\mathscr{A}^{\ast}$, resp. $\widetilde{\mathscr{A}}$) be a $d$-dimensional admissible family (resp. a dual family, resp. a co-dual family). Assume that $F_{\sigma}$ satisfied the hypothesis of \emph{Theorem \ref{indexbound}}, and
    let $\ens{A_k}_{k\in \N}\subset \mathscr{A}$ (resp. $\ens{A_k}_{k\in \N}\subset \mathscr{A}^{\ast}$, resp. $\ens{A_k}_{k\in \N}\subset \widetilde{\mathscr{A}}$) be a min-maximising sequence such that
	\begin{align*}
		\sup F_{\sigma_k}(A_k)\leq \beta(\sigma_k)+(\sigma_k-\sigma)
	\end{align*}
	and assume that all critical points of $F_{\sigma}$ in $K_{\beta(\sigma)}\cap A_{\infty}\cap \mathscr{E}(\sigma)$ are \emph{non-degenerate}. Then for all $0<\sigma\leq e^{-\frac{4}{\beta(0)}}$ such that $\beta$ satisfies the entropy condition at $\sigma$, there exists $x_{\sigma}\in K_{\beta(\sigma)}\cap A_{\infty}\cap \mathscr{E}(\sigma)$ \emph{(}resp. $x_{\sigma}^{\ast}\in K_{\beta^{\ast}(\sigma)}\cap A_{\infty}\cap \mathscr{E}(\sigma)$, resp. $\widetilde{x}_{\sigma}\in K_{\widetilde{\beta}(\sigma)}\cap A_{\infty}\cap \mathscr{E}(\sigma)$ \emph{)} such that  
	\begin{align}\label{concls}
	\left\{\begin{alignedat}{3}
		&F_{\sigma}(x_{\sigma})=\beta(\sigma),\quad&& \sigma^2G(x_{\sigma})\leq \frac{1}{\log(\frac{1}{\sigma})\log\log\left(\frac{1}{\sigma}\right)},\quad&& \text{and}\;\;
		\mathrm{Ind}_{F_{\sigma}}(x_{\sigma})\leq d.\\
		&F_{\sigma}(x_{\sigma}^{\ast})=\beta^{\ast}(\sigma),\quad && \sigma^2G(x_{\sigma}^{\ast})\leq \frac{1}{\log\left(\frac{1}{\sigma}\right)\log\log\left(\frac{1}{\sigma}\right)},\quad&& \text{and}\;\, \mathrm{Ind}_{F_{\sigma}}(x_{\sigma}^{\ast})\geq d\\
		&F_{\sigma}(\widetilde{x}_{\sigma})=\widetilde{\beta}(\sigma),\quad&& \sigma^2G(\widetilde{x}_{\sigma})\leq \frac{1}{\log\left(\frac{1}{\sigma}\right)\log\log\left(\frac{1}{\sigma}\right)},\quad&& \text{and}\;\, \mathrm{Ind}_{F_{\sigma}}(\widetilde{x}_{\sigma})=d
		\end{alignedat}\right.
	\end{align}
\end{theorem}
\begin{rem}
	Likewise, the proof would work equally well for homotopical and cohomotopical families, by Proposition \ref{vietoris}.
\end{rem}
\begin{proof}
	We give the proof in the special case where $X$ is $C^3$ and $F,G\in C^3(X,\R)$, in order to use Morse lemma as  in \cite{palais}. However, as the extension of the Morse lemma to $C^2$ spaces and functions (\cite{cambini}) is based on Cauchy-Lipschitz theorem and by the continuous dependence at the existence time with respect to the flow, the proof given below readily generalises to this weaker setting.
	
	Let $K= K_{\beta(\sigma)}\cap A_{\infty}\cap \mathscr{E}(\sigma)$ (notice that $K\neq \varnothing $ thanks of Theorem \ref{real}). As the critical points in $K$ are non-degenerate, $K$ is compact and consists of finitely many points $\ens{\bar{x}_0,\bar{x}_1,\cdots,\bar{x}_m}\subset K_{\beta(\sigma)}$. We cannot apply the previous lemma on $F_{\sigma}$ as the main lemma only work with $F_{\sigma_k}$. First by the Palais-Smale condition for $F_{\sigma}$ and as the critical points are isolated, we deduce that there exists $\delta>0$ such that $B_{2\delta}(x_i)\cap B_{2\delta}(x_j)=\varnothing$ for all $i\neq j$ and 
	\begin{align*}
		\norm{DF_{\sigma}(x)}\geq \delta\;\, \text{for all}\;\, x\in U_{2\delta}(K)\setminus U_{\delta}(K).
	\end{align*}
	Also notice that thanks of the proof of Theorem \ref{real}, for all $\ens{x_k}_{k\in \N}\subset X$ such that
	\begin{align}\label{s1}
		\left\{\begin{alignedat}{1}
			&\norm{DF_{\sigma_k}(x_k)}\conv{k\rightarrow \infty}0\\
			&|F_{\sigma_k}(x_k)-\beta(\sigma_k)|\conv{k\rightarrow \infty}0
		\end{alignedat}\right.
	\end{align}
	then
	\begin{align*}
		\left\{\begin{alignedat}{1}
			&\norm{DF_{\sigma}(x_k)}\conv{k\rightarrow \infty}0\\
			&|DF_{\sigma}(x_k)-\beta(\sigma)|\conv{k\rightarrow\infty}0
		\end{alignedat}\right.
	\end{align*}
	so up to a subsequence, we have thanks of the Palais-Smale condition for $F_{\sigma}$ at level $\beta(\sigma)$ that $x_k\conv{k\rightarrow \infty}x\in K_{\beta(\sigma)}$. In particular, if $\ens{x_k}\subset X$ verifies \eqref{s1}, then we can assume up to some relabelling that that for all $k\in \N$ large enough $x_k\in  N_{\delta}(\bar{x}_0)$. Now, looking at the proof of Morse Lemma by Palais (\cite{palais}) which only works for $C^3$ functions, we see that the diffeomorphism $\varphi$ around a critical point $\bar{x}_i$ such that
	\begin{align*}
		F_{\sigma}(\varphi_{\bar{x}_0}(x))=\beta(\sigma)+\norm{x_+}^2-\norm{x_-}^2
	\end{align*}
	is defined by 
	\begin{align*}
		\varphi_{\bar{x}_0}(x)=\sqrt{A_{\bar{x}_0}(\bar{x}_0)^{-1}A_{{x}_0}(x)}x,
	\end{align*}
	where for all $v,w\in H$, we have by Taylor expansion for some map $A_{\bar{x}_0}:B(\bar{x}_0,\delta)\rightarrow \mathscr{L}(H)$ with values into self-adjoint continuous operators
	\begin{align*}
		&F_{\sigma}(x)=\beta(\sigma)+\s{A_{\bar{x}_0}(x)x}{x}\\
		&D^2F_{\sigma}(\bar{x}_0)(v,w)=2\s{A_{\bar{x}_0}(\bar{x}_0)v}{w}
	\end{align*}
	Now, notice if $B_{\bar{x}_0}(x)=A_{\bar{x}_0}(\bar{x}_0)^{-1}A_{\bar{x}_0}(x)$ that $B_{\bar{x}_0}(\bar{x}_0)=\mathrm{Id}_{H}$ and as for $\norm{h}<1$, $\sqrt{\mathrm{Id}+h}$ is well defined by the absolutely convergent series
	\begin{align*}
		\sqrt{\mathrm{Id}+h}=\sum_{n=0}^{\infty}\binom{\frac{1}{2}}{n}h^n,
	\end{align*}
	we deduce that for some $\delta_0>0$ small enough and depending only on $A_{\bar{x}_0}$, namely such that for all $x\in B(x,\delta)$
	\begin{align}\label{s2}
		\norm{x-B_{\bar{x}_0}(x)}<\frac{1}{2}\;\,\text{for all}\;\, x\in B(\bar{x}_0,\delta)\quad (1\;\, \text{would be enough})
	\end{align}
	that $\varphi(x)$ is well-defined on $B(\bar{x}_0,\delta)$ and $C^1$. Therefore, thanks of the local inversion theorem, up to diminishing $\delta$, we can assume that $\varphi$ is a diffeomorphism from $B(\bar{x}_0,\delta)$ onto its image (here, $\delta$ depends only on $A_{\bar{x}_0}$). 
	
	Now, let $x_k\in K_{\beta(\sigma_k)}$ be a critical point of $F_{\sigma_k}$ and $A_{x_k}$ such that 
	\begin{align*}
		&F_{\sigma_k}(x)=\beta(\sigma_k)+\s{A_{x_k}(x)x}{x}\\
		&DF^2_{\sigma_k}(x)=2\s{A_{x_k}(x_k)x}{x}
	\end{align*}
	such that $x_k\conv{k\rightarrow \infty}\bar{x}_0$. Thanks of the strong convergence, we deduce that for $k$ large enough, $A_{x_k}(x_k)$  is an invertible operator so we can define for $k$ large enough $B_{x_k}(x)=A_{x_k}(x_k)^{-1}A_{x_k}(x)$. Now, taking $k$ large enough such that $B(x_k,\frac{\delta}{2})\subset B(\bar{x}_0,\delta)$, we see by the strong convergence of $x_k\rightarrow \bar{x}_0$ that
	\begin{align*}
		\norm{B_{x_k}-B_{\bar{x}_0}}_{B(x_k,\frac{\delta}{2})}\conv{k\rightarrow \infty}0.
	\end{align*}
	In particular, if $k$ is large enough such that
	\begin{align*}
		\norm{B(x_k)(x)-B_{\bar{x}_0}(x)}\leq \frac{1}{2}\;\, \text{for all}\;\, x\in B\left(x_k,\frac{\delta}{2}\right),
	\end{align*}
	we deduce by \eqref{s2} that
	\begin{align*}
		\norm{x-B_{x_k}(x)}<1\;\, \text{for all}\;\, x\in B\left(x_k,\frac{\delta}{2}\right).
	\end{align*}
	In particular, we can define $\varphi_{x_k}(x)=\sqrt{B_{x_k}(x)}x$ for all $x\in B(x,\frac{\delta}{2})$, and we see that in particular $d\varphi_k(x_k)=\mathrm{Id}$. Now, as 
	\begin{align*}
		\norm{\varphi_{x_k}-\varphi_{\bar{x}_0}}_{C^1(B(x_k,\frac{\delta}{2}))}\conv{k\rightarrow\infty}0,
	\end{align*}
	and as the neighbourhood around which $\varphi_{x_k}$ is invertible depends only on the local behaviour of its derivative around $x_k$ and as $\varphi_{\bar{x}_0}$ is invertible in $B(\bar{x}_0,\delta)$, we deduce that for $k$ large enough, $\varphi_k$ is invertible on $B(x_k,\frac{\delta}{4})$, so the Morse lemma implies that
	\begin{align*}
		F_{\sigma_k}(\varphi_k(x))=\beta(\sigma_k)+\norm{x_+^k}^2-\norm{x^k_-}^2\;\,\text{for all}\;\, x\in B\left(x_k,\frac{\delta}{4}\right)
	\end{align*}
	In particular, $F_{\sigma_k}$ has only one critical point on $B(\bar{x}_0,\frac{\delta}{8})\subset B(x_k,\frac{\delta}{4})$ for $k$ large enough. Therefore, we can apply the Proposition \ref{ext} to $F_{\sigma_k}$ with $\delta>0$ and $\epsilon>0$ \emph{independent of $k$}.
	
	As $K=\ens{\bar{x}_0,\bar{x}_1,\cdots,\bar{x}_m}$ is finite, we saw that for all $k$ sufficiently large, $F_{\sigma_k}$ has at most one critical point in $B(x_i,\frac{\delta}{8})$. Let us denote by $K_{\beta(\sigma_k)}\cap U_{\delta/8}(K)=\ens{x_0^k,x_1^k,\cdots,x_{m_k}^k}$ where $m_k\leq m$ the critical points of $F_{\sigma_k}$ at level $\beta(\sigma_k)$. Thanks of Proposition \ref{ext} and the first part of the proof, there exists some $\delta>0$ independent of $k$ such that for all ${A}\in \mathscr{A}$ such that $\sup F_{\sigma_k}(A)\leq \beta(\sigma)+\delta$, then for all $1\leq i\leq m_{k}$, there exists $A'_i\in \mathscr{A}$ 
	such that
	\begin{align}\label{p3}
		\left\{\begin{alignedat}{1}
			&A\setminus U_{2\epsilon}(x_i^k)=A_i'\setminus U_{2\epsilon}(x_i^k)\\
			&A_i'\cap U_{\epsilon}(x_i^k)=\varnothing\\
			&\sup F_{\sigma_k}(A_i')\leq \sup F_{\sigma_k}(A).
		\end{alignedat}\right.
	\end{align}
	Furthermore, as the $x_i^k$ are uniformly isolated independently of $k$, taking $\epsilon>0$ small enough, we can assume that
	\begin{align}
		&U_{2\epsilon}(x_i^k)\cap U_{2\epsilon}(x_j^k)=\varnothing\quad  \text{for all }\;\,1\leq i\neq j\leq m_k \label{2.24}.
	\end{align}
	and that if $x_i^k\conv{k\rightarrow \infty}x_j\in K$ (for some $j\in \ens{1,\cdots,m}$) that $k$ is large enough such that 
	\begin{align*}
		U_{\frac{\epsilon}{2}}(x_j)\subset U_{\epsilon}(x_i^k).
	\end{align*}
	We also remark that  
	\begin{align}\label{defu}
		U_{\frac{\epsilon}{2}}(K)=\bigcup_{i=1}^nU_{\frac{\epsilon}{2}}(x_i)
	\end{align}
	is an open neighbourhood of $K=K_{\beta(\sigma)}\cap A_{\infty}\cap \mathscr{E}(\sigma)$.
	Now, let $\ens{\sigma_k}\subset (\sigma,\infty)$ such that $\sigma_k\conv{k\rightarrow \infty}\sigma$ and $\ens{A_k}_{k\in \N}$ such that
	\begin{align*}
		\sup F_{\sigma_k}(A_k)\leq \beta(\sigma_k)+(\sigma_k-\sigma)\conv{k\rightarrow'\infty}\beta(\sigma).
	\end{align*}
	In particular, there exists $k_0\in \N$ such that for all $k\geq k_0$, there holds
	\begin{align*}
		\sup F_{\sigma_k}(A_k)\leq \beta(\sigma)+\delta.
	\end{align*}
	We can also assume as $K$ is isolated in $K_{\beta(\sigma)}\cap \mathscr{E}(\sigma)$ and thanks of the first part of the proof that $\epsilon>0$ is small enough such that
	\begin{align}\label{pn}
		A_k\cap U_{\frac{\epsilon}{2}}\left((K_{\beta(\sigma)}\cap \mathscr{E}(\sigma))\setminus K\right)=\varnothing\;\, \text{for all}\;\, k \;\, \text{large enough}.
	\end{align}
	Now, define by induction a finite sequence (recall that $m_k\leq m$) $A^0_k,A_k^1,\cdots,A_k^{m_k}\in \mathscr{A}$ by $A_k^0=A_k$, $A_k^1=(A_k^0)_1'=(A_k)_0'$, 
	\begin{align*}
		A_k^j=(A_k^{j-1})'_j
	\end{align*}
	using the notation of \eqref{p3}. We see in particular that by \eqref{p3} and \eqref{2.24}
	\begin{align}\label{com1}
		A_k^j\cap U_{\epsilon}(x_i^k)=\varnothing\;\, \text{for all}\;\, 1\leq i\leq j\leq m_k.
	\end{align}
	Furthermore, as for all $1\leq  j\leq m_k$, we have by \eqref{p3}
	\begin{align}\label{com2}
		\sup F_{\sigma_k}(A_k^j)\leq \sup F_{\sigma_k}(A_k^{j-1})
	\end{align}
	so by combining \eqref{pn}, \eqref{defu} with \eqref{com1} and \eqref{com2}, we deduce that for all $k\geq k_0$, we have
	\begin{align}\label{contradiction}
		&A^{m_k}_k\cap U_{\frac{\epsilon}{2}}(K_{\beta(\sigma)}\cap \mathscr{E}(\sigma))=\varnothing\\
		& \beta(\sigma)\leq \sup F_{\sigma_k}(A_k^{m_k})\leq \sup F_{\sigma_k}(A_k)\leq \beta(\sigma_k)+(\sigma_k-\sigma)\conv{k\rightarrow \infty}\beta(\sigma).
	\end{align}
	By Theorem \ref{real} there exists a sequence $\ens{x_k}_{k\in \N}\subset X$ such that
	\begin{align}\label{cont}
		\left\{\begin{alignedat}{1}
			(1)\;\,&\mathrm{dist}(x_k,A_k^{m_k})\conv{k\rightarrow \infty}0\\
			(2)\;\,&\beta(\sigma)-(\sigma_k-\sigma)\leq F_{\sigma}(x_k)\leq F_{\sigma_k}(x_k)\leq \beta(\sigma_k)+(\sigma_k-\sigma)\\
			(3)\;\,&\norm{DF_{\sigma}(x_k)}\conv{k\rightarrow \infty}0\\
			(4)\;\,&\inf_{k\in \N}F(x_k)>0.
		\end{alignedat}\right.
	\end{align}
	Therefore, by the Palais-Smale condition at level $\beta(\sigma)$ and \eqref{cont}, up to a subsequence we have $x_k\conv{k\rightarrow \infty}x_{\infty}\in K_{\beta(\sigma)}\cap \mathscr{E}(\sigma)$. However, we have for all $k$ large enough by \eqref{contradiction} and as $\mathrm{dist}(x_k,A_k^{m_k})\conv{k\rightarrow \infty}0$
	\begin{align*}
		\mathrm{dist}(A_k^{m_k},K_{\beta(\sigma)}\cap \mathscr{E}(\sigma))\geq \frac{\epsilon}{4},
	\end{align*}
	and this contradicts the fact that $x_{\infty}\in K_{\beta(\sigma)}\cap \mathscr{E}(\sigma)$. This concludes the proof of the theorem.
\end{proof}

\subsection{Marino-Prodi perturbation method and the degenerate case}

Let us recall the main theorem here.

\begin{theorem}\label{indexbound2}
	Let $(X,\norm{\,\cdot\,}_X)$ be a $C^2$ Finsler manifold modelled on a Banach space $E$, and $Y\hookrightarrow X$ be a $C^2$ Finsler-Hilbert manifold modelled on a Hilbert space $H$ which we suppose locally Lipschitz embedded in $X$, 
	and let $F,G\in C^2(X,\R_+)$ be two fixed functions. Define for all $\sigma>0$, $F_\sigma=F+\sigma^2G\in C^2(X,\R_+)$ and suppose that the following conditions hold.
	\begin{enumerate}
		\item[\emph{(1)}] \emph{\textbf{Palais-Smale condition:}} For all $\sigma>0$, the function $F_{\sigma}:X\rightarrow Y$ satisfies the Palais-Smale condition at all positive level $c>0$.
		\item[\emph{(2)}] \emph{\textbf{Energy bound:}} The following energy bound condition holds : for all $\sigma>0$ and for all $\ens{x_k}_{k\in \N}\subset X$ such that 
		\begin{align*}
		\sup_{k\in \N}F_{\sigma}(x_k)<\infty,
		\end{align*}
		we have
		\begin{align*}
		\sup_{k\in \N}\norm{\D G(x_k)}<\infty.
		\end{align*}
		\item[\emph{(3)}] \textbf{\emph{Fredholm property:}} For all $\sigma>0$ and for all $x\in K(F_{\sigma})$, we have $x\in Y$, and the second derivative $D^2F_{\sigma}(x):T_xX\rightarrow T_x^{\ast}X$ restrict on the Hilbert space $T_xY$ such that the linear map $\D^2F_{\sigma}(x)\in \mathscr{L}(T_yY)$ defined by
		\begin{align*}
		D^2F_{\sigma}(x)(v,v)=\s{\D^2F_{\sigma}(x)v}{v}_{Y,x},\quad \text{for all}\;\, v\in T_xY,
		\end{align*}
		is a \emph{Fredholm} operator, and the embedding $T_xY\hookrightarrow T_xX$ is dense for the Finsler  norm $\norm{\,\cdot\,}_{X,x}$.
	\end{enumerate}
	Now, let $\mathscr{A}$ \emph{(}resp. $\mathscr{A}^{\ast}$, resp. $\bar{\mathscr{A}}$, resp. $\mathscr{A}(\alpha_{\ast})$, resp. $\mathscr{A}(\alpha^{\ast})$, where the last two families depend respectively on a homology class $\alpha_{\ast}\in H_d(Y,B)$ - where $B\subset Y$ is a fixed compact subset -  and a cohomology class $\alpha^{\ast}\in H^d  (Y)$\emph{)} be a $d$-dimensional admissible family of $Y$ (resp. a $d$-dimension dual family to $\mathscr{A}$, resp. a $d$-dimensional co-dual family to $\mathscr{A}$, resp. a $d$-dimensional homological family, resp. a $d$-dimensional co-homological family) with boundary $\ens{C_i}_{i\in I}\subset Y$.
	Define for all $\sigma>0$
	\begin{align*}
	\begin{alignedat}{3}
	&\beta(\sigma)=\inf_{A\in\mathscr{A}}\sup F_\sigma(A)<\infty,\quad &&
	\beta^{\ast}(\sigma)=\inf_{A\in \mathscr{A}^{\ast}}\sup F_{\sigma}(A),
	\quad&&\widetilde{\beta}(\sigma)=\inf_{A\in \widetilde{\mathscr{A}}}\sup F_{\sigma}(A)\\
	&\bar{\beta}(\sigma)=\inf_{A\in \mathscr{A}(\alpha_{\ast})}\sup F_{\sigma}(A),\quad&& 
	\underline{\beta}(\sigma)=\inf_{A\in \mathscr{A}(\alpha^{\ast})}\sup F_{\sigma}(A).
	\end{alignedat}
	\end{align*} 
	Assuming that the min-max value is non-trivial, \textit{i.e.}
	\begin{enumerate}\label{non-trivial}
		\item[\emph{(4)}] \textbf{\emph{Non-trivialilty:}} $\displaystyle \beta_0=\inf_{A\in \mathscr{A}}\sup F(A)>\sup_{i\in I} \sup F(C_i)=\widehat{\beta}_0$, 
	\end{enumerate}
	there exists a sequence $\ens{\sigma_k}_{k\in \N}\subset (0,\infty)$ such that $\sigma_k\conv{k\rightarrow \infty}0$, and for all $k\in \N$, there exists a critical point $x_k\in K(F_{\sigma_k})\in \mathscr{E}(\sigma_k)$ \emph{(}resp. $x_k^{\ast},\widetilde{x}_k,\bar{x}_{k},\underline{x}_k\in \mathscr{E}(\sigma_k)$\emph{)} of $F_{\sigma_k}$ satisfying the entropy condition \eqref{boltzmann} and such that respectively
	\begin{align*}
	\left\{
	\begin{alignedat}{2}
	&F_{\sigma_k}(x_k)=\beta(\sigma_k),\quad&&
	\mathrm{Ind}_{F_{\sigma_k}}(x_k)\leq d\\
	&F_{\sigma_k}(x_k^{\ast})=\beta^{\ast}(\sigma_k),\quad&&
	\mathrm{Ind}_{F_{\sigma_k}}(x_{k})\geq d		\\
	&F_{\sigma_k}(\widetilde{x}_k)=\widetilde{\beta}(\sigma_k),\quad&&
	\mathrm{Ind}_{F_{\sigma_k}}(\widetilde{x}_k)\leq d\leq \mathrm{Ind}_{F_{\sigma_k}}(\widetilde{x}_k)+\mathrm{Null}_{F_{\sigma_k}}(\widetilde{x}_k)
	\\
	&F_{\sigma_k}(\bar{x}_k)=\bar{\beta}(\sigma_k),\quad&&
	\mathrm{Ind}_{F_{\sigma_k}}(\bar{x}_k)\leq d\leq \mathrm{Ind}_{F_{\sigma_k}}(\bar{x}_k)+\mathrm{Null}_{F_{\sigma_k}}(\bar{x}_k)
	\\
	&F_{\sigma_k}(\underline{x}_k)=\underline{\beta}(\sigma_k),\quad&&
	\mathrm{Ind}_{F_{\sigma_k}}(\underline{x}_k)\leq d\leq \mathrm{Ind}_{F_{\sigma_k}}(\underline{x}_k)+\mathrm{Null}_{F_{\sigma_k}}(\underline{x}_k).
	\end{alignedat}\right.
	\end{align*}
\end{theorem}

\begin{proof}
	As we have mentioned already, we can assume that $X$ is a Finsler-Hilbert manifold modelled on a Hilbert space $H$. Take $\sigma>0$ such that $\beta$ satisfies the entropy condition at $\sigma$. If $F_{\sigma}$ has only non-degenerate critical points in $K_{\beta(\sigma)}\cap \mathscr{E}(\sigma)$, then we are done.
	
	\begin{lemme}\label{t}
		Let $\ens{a_j}_{j\in \N}\subset [0,\infty)$ and $\ens{b_j}_{j\in \N}\subset (0,\infty)$ be two sequences such that
		\begin{align*}
			\sum_{j\in \N}^{}a_j<\infty\quad \text{and}\;\, \sum_{j\in \N}^{}b_j=\infty.
		\end{align*}
		Then there holds
		\begin{align*}
			\liminf_{j\rightarrow \infty}\frac{a_j}{b_j}=0.
		\end{align*}
	\end{lemme}
\begin{proof}
	By contradiction, let $\delta>0$ such that
	\begin{align*}
		\liminf\limits_{j\rightarrow \infty}\frac{a_j}{b_j}=\delta.
	\end{align*}
	Then there exists $J\in \N$ such that for all $j\geq J$, 
	\begin{align*}
		\frac{a_j}{b_j}\geq \frac{\delta}{2},
	\end{align*}
	so that for all $j\geq J$, there holds $\delta\, b_j\leq a_j$. Therefore, we obtain
	\begin{align*}
		\sum_{j\geq J}b_j\leq \frac{1}{\delta}\sum_{j\geq J}^{}a_j<\infty,
	\end{align*}
	contradicting the divergence of $\sum b_j$.
\end{proof}

	Let $\ens{a_j}_{j\in\N}\subset (0,1)$ be a strictly decreasing sequence  converging to zero. Then there holds as $\beta$ is increasing for all $j\in\N$
	\begin{align*}
		\int_{a_{j+1}}^{a_j}\beta'(\sigma)d\sigma\leq \beta(a_j)-\beta(a_{j+1})
	\end{align*}
	and we notice that
	\begin{align*}
		\sum_{j=0}^{n}\left(\beta(a_j)-\beta(a_{j+1})\right)= \beta(a_0)-\beta(a_{n+1})\conv{n\rightarrow\infty}\beta(a_0)-\beta(0)<\infty.
	\end{align*}
	This implies that
	\begin{align*}
		\sum_{j\in \N}^{}\left(\beta(a_j)-\beta(a_{j+1})\right)<\infty.
	\end{align*}
	Therefore, if $b=\ens{b_j}_{j\in\N}$ is a the general term of a divergent series with positive terms,  there exists by Lemma \ref{t} a subsequence $\ens{j_l}_{l\in \N}$ such that for all $l\in \N$, there holds
	\begin{align}\label{ineq}
		\beta(a_{j_l})-\beta(a_{j_l+1})\leq b_{j_l}.
	\end{align}
	Now, for convenience of notation, as we do not use any properties related to the convergence of the series of general term $\ens{b_{j_l}}_{l\in \N}$, we will assume that \eqref{ineq} holds for all $j\in \N$.
	Now, we want to find such sequence $\ens{a_j}_{j\in \N}$ and $\ens{b_j}_{j\in \N}$ such that
	\begin{align*}
		(a_j-a_{j+1})^{-1}b_j\leq \frac{1}{a_j\log\left(\frac{1}{a_j}\right)\log\log\left(\frac{1}{a_j}\right)\log\log\log\left(\frac{1}{a_j}\right)}
	\end{align*}
	Take $a_j=\dfrac{1}{j}$, we have $a_j-a_{j+1}=\dfrac{1}{j(j+1)}$, so the condition becomes for $j\geq 4\cdot 10^{6}>e^{e^e}$
	\begin{align}\label{rhs2}
		b_j\leq \frac{1}{(j+1)\log(j)\log\log(j)\log\log\log(j)}, 
	\end{align}
	and the series whose general term is the right-hand side of \eqref{rhs2} diverges so we define $\ens{b_j}_{j\in \N}\subset (0,\infty)^{\N}$ such that for all $j\geq J\geq 4\cdot 10^{6}>e^{e^e}$
	\begin{align*}
		b_j=\frac{1}{(j+1)\log(j)\log\log(j)\log\log\log(j)}.
	\end{align*}
	Now, for all $j\geq J$, let $I_j=[a_{j+1},a_j]$ and
	\begin{align*}
		A_j=I_j\cap\ens{\sigma:\beta'(\sigma)\leq \frac{1}{a_j\log\left(\frac{1}{a_j}\right)\log\log\left(\frac{1}{a_j}\right)}\leq \frac{1}{\sigma\log\left(\frac{1}{\sigma}\right)\log\log\left(\frac{1}{\sigma}\right)}},
	\end{align*}
	and define $\delta_j$ for $j\geq J$ by
	\begin{align*}
		\delta_j=\frac{1}{\log\log\log(j)}\conv{j\rightarrow \infty}0.
	\end{align*}
	Then for all $j\geq J$, there holds by \eqref{ineq}
	\begin{align*}
		\int_{a_{j+1}}^{a_j}\beta'(\sigma)d\sigma\leq \frac{\delta_j(a_j-a_{j+1})}{a_j\log(\frac{1}{a_j})\log\log(\frac{1}{a_j})}
	\end{align*}
	so that
	\begin{align*}
		\frac{\mathscr{L}^1(I_j\setminus A_j)}{a_j\log\left(\frac{1}{a_j}\right)\log\log\left(\frac{1}{a_j}\right)}\leq \int_{I_j\setminus A_j}\beta'(\sigma)d\sigma\leq \int_{I_j}\beta'(\sigma)d\sigma\leq \frac{\delta_j\mathscr{L}^1(I_j)}{a_j\log\left(\frac{1}{a_j}\right)\log\log\left(\frac{1}{a_j}\right)}
	\end{align*}
	so that
	\begin{align}\label{entropy3}
		\frac{\mathscr{L}^1(I_j\setminus A_j)}{\mathscr{L}^1(I_j)}\leq \delta_j\conv{j\rightarrow \infty}0.
	\end{align}
	Therefore, we obtain for all $j\geq J$ some element $\sigma_j\in (a_{j+1},a_j)$ such that
	\begin{align*}
		\beta'(\sigma_j)\leq \frac{\beta(a_j)-\beta(a_{j+1})}{a_j-a_{j+1}}\leq \frac{1}{a_j\log(\frac{1}{a_j})\log\log(\frac{1}{a_j})}\leq \frac{1}{\sigma_j\log(\frac{1}{\sigma_j})\log\log(\frac{1}{\sigma_j})}.
	\end{align*}
	Now, for all $\sigma \in (0,1)$, as $K(F_{\sigma})$ is compact, we let $\varphi_{\sigma}$ be the cut-off function given by Proposition \ref{marinoprodi} and let $\epsilon(\sigma)>0$ such that for all $\norm{y}< \epsilon(\sigma)$ small enough such that by Proposition \ref{propre}, the map
	\begin{align}\label{coff}
		F_{\sigma,y}=F_{\sigma}+\varphi_{\sigma}\s{y}{\,\cdot\,}
	\end{align}
	is proper on $N_{2\delta}(K)$. Now, fix some $C>0$. 
	
	\textbf{Claim 1}:  there exists $\delta(C,\sigma)>0$ (taken such that $\delta(C,\sigma)<\epsilon(\sigma)$) such that for all $|\tau-\sigma|<\delta(\sigma)$, the map 
	\begin{align*}
		F^{\tau}_{\sigma,y}=F_{\sigma,y}+(\tau^2-\sigma^2)G=F_{\tau}+\varphi_{\sigma}\s{y}{\,\cdot\,}
	\end{align*}
	is such that
	\begin{align}\label{inclusion}
		K(F^{\tau}_{\sigma,y})\cap\ens{x:F_{\sigma,y}^{\tau}(x)\leq C}\subset K(F_{\sigma})^{\delta},
	\end{align}
	and that for all $y$ such that $F_{\sigma,y}$ is non-degenerate (such $y$ form a non-meager subset by Sard-Smale theorem and Proposition \ref{marinoprodi}), the map $F^{\tau}_{\sigma,y}$ has only non-degenerate critical points below the critical level $C>0$ (in practise we can just take $C=\beta(1)+1$, but $C=\beta(\sigma_0)+\eta$ for some $\sigma_0,\eta>0$ would work equally well).

	First, observe that $F_{\sigma,y}$ has no critical points in $X\setminus K(F_{\sigma})^{2\delta}$, as $F_{\sigma,y}=F_{\sigma}$ in $X\setminus K(F_{\sigma})^{2\delta}$. Now, by contradiction, assume that  there exists $\ens{\tau_k}_{k\in \N}$ such that $\tau_k\rightarrow \sigma$ and a sequence of critical points $\ens{x_k}_{k\in \N}\subset X$ (\textit{i.e.} such that $x_k\in K(F_{\sigma,y}^{\tau_k})\cap\ens{x:F_{\sigma}^{\tau}(x)\leq C}$ for all $k\in \N$) and
	\begin{align*}
		\mathrm{dist}(x_k,K(F_{\sigma}))\geq \delta.
	\end{align*}
	Then, by the same proof \textit{mutadis mutandis} of $(6.9)$ of Proposition $6.3$ in \cite{geodesics}, we have thanks of the condition $(2)$ on the energy bound that
	\begin{align*}
		\norm{\D F_{\sigma,y}^{\tau_k}(x_k)-\D F_{\sigma,y}}=(\tau_k^2-\sigma^2)\norm{\D G(x_k)}\conv{k\rightarrow \infty}0
	\end{align*}
	Now, if $x_k\in K(F_{\sigma})^{2\delta}$ for $k$ large enough, as $\D F_{\sigma,y}$ is proper on $K(F_{\sigma})^{2\delta}$, we deduce that up to a subsequence, we have 
	\begin{align*}
		x_k\conv{k\rightarrow \infty}x_{\infty}\in K(F_{\sigma,y})\subset K(F_{\sigma})^{\delta},
	\end{align*}
	a contradiction. Therefore, we can assume that 
	for all $k\in \N$) and
	\begin{align*}
		\mathrm{dist}(x_k,K(F_{\sigma}))\geq 2\delta.
	\end{align*}
	Furthermore, as $F^{\tau_k}_{\sigma,y}=F_{\tau_k}$ and $F_{\sigma,y}=F_{\sigma}$ on $X\setminus K(F_{\sigma})^{2\delta}$,  we have
	\begin{align*}
		F_{\sigma}(x_k)\leq F_{\tau_k}(x_k)=F^{\tau_k}_{\sigma,y}(x_k)\leq C.
	\end{align*}
	Therefore, by the Palais-Smale condition for $F_{\sigma}$, we deduce that up to subsequence, we have $x_k\conv{k\rightarrow \infty}x_{\infty}\in K(F_{\sigma})$, a contradiction. 
	Now, to prove the second part of the claim, by \eqref{inclusion}, if $\tau_k\conv{k\rightarrow \infty}\sigma$ and $\ens{x_k}_{k\in \N}\subset X$ is a sequence of critical points  associated to $\ens{F^{\tau_k}_{\sigma,y}}_{k\in \N}$ such that
	\begin{align*}
		F_{\sigma,y}^{\tau_k}(x_k)\leq C,
	\end{align*}
	we have by properness of $F_{\sigma,y}$ on $K(F_{\sigma})^{2\delta}$ that (up to a subsequence)  $x_k\conv{k\rightarrow \infty}x_{\infty}\in K(F_{\sigma,y})$. Furthermore, the strong convergence of $\ens{x_k}_{k\in \N}$ towards $x_{\infty}$ shows that
	\begin{align}\label{convnd}
		\norm{\D^2F^{\tau_k}_{\sigma,y}(x_k)-\D^2F_{\sigma,y}(x_{\infty})}\conv{k\rightarrow \infty}0,
	\end{align} 
	as we see these two second order operators defined on the underlying Hilbert space $H\simeq T_{x_k}X\simeq T_{x_{\infty}}X$.
	Now, we recall the following continuity property of the spectrum for bounded linear operators on a Hilbert space $H$, which we state below.
	\begin{enumerate}
	 \item[(P)] For all $T\in \mathscr{L}(H)$, for all $\epsilon>0$, there exists $\delta>0$ such that for all $S\in \mathscr{L}(H)$ such that $\norm{T-S}<\delta$, there holds $\mathrm{Sp}(S)\subset U_{\epsilon}(\mathrm{Sp}(T))$,
	\end{enumerate}
	  where $\mathrm{Sp}(T)\subset \R$ (resp. $\mathrm{Sp}(S)\subset \R$) is the spectrum of $T$ (resp. $S$) and $U_{\epsilon}(\mathrm{Sp}(T))$ is the $\epsilon$-neighbourhood in $\R$ of the compact subset $\mathrm{Sp}(T)\subset \R$. Now, as $0\notin \mathrm{Sp}(\D^2F_{\sigma,y}(x_{\infty}))$, and $\mathrm{Sp}(\D^2F_{\sigma,y}(x_{\infty}))\subset \R$ is compact, there exists $\epsilon>0$ such that 
	\begin{align}\label{p2}
		\mathrm{Sp}(\D^2F_{\sigma,y}(x_{\infty}))\cap (-2\epsilon,2\epsilon)=\varnothing.
	\end{align}
	Thanks of \eqref{convnd}, for $k$ large enough, we have by \eqref{p2}
	\begin{align*}
		\mathrm{Sp}\left(\D^2F_{\sigma,y}^{\tau_k}(x_k)\right)\subset U_{\epsilon}\left(\mathrm{Sp}(\D^2F_{\sigma,y}(x_{\infty}))\right)\subset \R\setminus (-\epsilon,\epsilon)
	\end{align*}
	so that in particular $0\notin \mathrm{Sp}\left(\D^2F_{\sigma,y}^{\tau_k}(x_k)\right)$, and $F^{\tau_k}_{\sigma,y}$ is non-degenerate.
	
	Finally, as $F_{\sigma,y}$ has a finite number of critical points, and all of them are non-degenerate, this argument can be made uniform in $x_{\infty}$ and this complete the proof of the claim.
	
	\textbf{Important remark:} As $F^{\tau}_{\sigma,y}-F_{\sigma,y}=F_{\tau}-F_{\sigma}$ which is independent of $y$, the value $\delta(C,\sigma)$ found previously is independent of $y$ sufficiently small. 
	
	Now, we fix some $C>\beta(1)$ and for all $\sigma\in (0,1)$, we denote $\delta(\sigma)=\delta(C,\sigma)$, and we observe that for all $j\geq J$, there holds
	\begin{align*}
		I_j=\bigcup_{\sigma\in I_j}B(\sigma,\delta(\sigma)).
	\end{align*}
	Therefore, by compactness of $I_j=[a_{j+1},a_j]\subset \R$, there exists $N_j\in \N$ and $\sigma_1,\cdots,\sigma_{N_j}\in I_j$ such that
	\begin{align*}
		I_j\subset \bigcup_{i=1}^{N_j}B(\sigma,\delta(\sigma_i)),
	\end{align*}
	and up to relabelling, we can assume that $a_{j+1}\leq \sigma_1<\sigma_2<\cdots<\sigma_{N_j}\leq a_j$. In particular, must have in particular $\sigma_{i+1}-\sigma_i<\delta(\sigma_i)$ for all $1\leq i\leq N_{j}-1$, while
	$\sigma_1-a_{j+1}<\delta(\sigma_1)$, and $a_{j}-\sigma_{N_j}<\delta_{N_j}$.
	Therefore, we define for all $y\in X$,
	\begin{align}\label{enddef}
		\widetilde{F}_{\sigma,y}=F_{\sigma}+\mathbf{1}_{\ens{a_j\leq \sigma<\sigma_1}}\varphi_{\sigma_1}\s{y}{\cdot}+\sum_{i=1}^{N_j-1}\mathbf{1}_{\ens{\sigma_i\leq \sigma< \sigma_{i+1}}}\varphi_{\sigma_i}\s{y}{\,\cdot\,}+\mathbf{1}_{\ens{\sigma_{N_j}\leq \sigma\leq a_j}}\varphi_{\sigma_{N_j}}\s{y}{\,\cdot\,}
	\end{align}
	where $\varphi_{\sigma_i}$ is the cut-off given by \eqref{coff} from Proposition \ref{marinoprodi}. Now, by notational convenience, we let $\sigma_0=a_{j+1}$ and $\sigma_{N_j+1}=a_j$.
	
	Now, notice that $\beta(\sigma,y)=\beta(\widetilde{F}_{\sigma,y},\mathscr{A})$ is increasing on each interval $[\sigma_i,\sigma_{i+1}]$ for all $0\leq i\leq N_{j}+1$. For all $y\in X$ such that $\Vert y\Vert \leq \epsilon$, we have
	\begin{align*}
		\norm{F_{\sigma}-F_{\sigma,y}}_{C^2(X)}\leq \epsilon.
	\end{align*}
	Therefore, up to replacing $\s{y}{\,\cdot\,}$ by $\s{y}{\,\cdot-x_i}$ for some $x_i\in X$ in each component of the sum on the right-hand side of \eqref{enddef} we can assume that $F_{\sigma,y}\geq F_{\sigma}$, so that
	\begin{align*}
		\beta(\sigma)\leq \beta(\sigma,y)\leq \beta(\sigma)+\epsilon,
	\end{align*}
	and this property of $\sigma \mapsto\beta(\sigma,y)$ implies that
	\begin{align*}
		&\int_{a_{j+1}}^{a_j}\beta'(\sigma,y)d\sigma=\sum_{i=0}^{N_j}\int_{\sigma_i}^{\sigma_{i+1}}\beta'(\sigma,y)d\sigma\leq \sum_{i=0}^{N_j}\left(\beta(\sigma_{i+1},y)-\beta(\sigma_i,y)\right)\leq \sum_{i=0}^{N_j}\left(\beta(\sigma_{i+1})-\beta(\sigma_i)+\epsilon\right)\\
		&=\beta(a_j)-\beta(a_{j+1})+(N_j+1)\epsilon.
	\end{align*}
	Taking
	\begin{align*}
		\epsilon\leq \frac{b_j}{N_j+1}, 
	\end{align*}
	implies that the set 
	\begin{align*}
		A_j(y)=I_j\cap\ens{\sigma:\beta'(\sigma,y)\leq \frac{1}{a_j\log\left(\frac{1}{a_j}\right)\log\log\left(\frac{1}{a_j}\right)}\leq \frac{1}{\sigma\log\left(\frac{1}{\sigma}\right)\log\log\left(\frac{1}{\sigma}\right)}},
	\end{align*}
	verifies by \eqref{entropy3}
	\begin{align*}
		\mathscr{L}^1(A_j)\geq (1-2\delta_j)\mathscr{L}^1(I_j).
	\end{align*}
	In particular, there exists for $j\geq 6\cdot 10^{702}>e^{e^{e^2}}$ an element $\sigma(y)\in I_j$ such that $\widetilde{F}_{\sigma,y}$ verifies the entropy condition at $\sigma(y)$. Furthermore, as $\sigma(y)\in B(\sigma_i,\delta(\sigma_i))$ for some $i\in \ens{1,\cdots,N_j}$, we deduce that $\widetilde{F}_{\sigma(y),y}=F_{\sigma_i,y}$ is non-degenerate and is proper on an open neighbourhood of its critical set at level $\beta(\sigma(y),y)$, so verifies the Palais-Smale condition at this level (recall that $F_{\sigma(y),y}=F_{\sigma_{i},y}$ for some $i\in \ens{0,\cdots, N_j}$, so these properties hold by \textbf{Claim 1}). Furthermore, as 
	\begin{align*}
		\frac{d}{d\sigma}F_{\sigma,y}=\frac{d}{d\sigma}F_{\sigma},
	\end{align*}
	we obtain by Theorem \ref{ndeg} a critical point $x_y\in X$ of $\widetilde{F}_{\sigma(y),y}$ such that 
	\begin{align}\label{end}
		\widetilde{F}_{\sigma(y),y}(x(y))=\beta(\sigma,y), \quad \sigma(y)^2G(\sigma(y))\leq \frac{1}{\log\left(\frac{1}{\sigma(y)}\right)\log\log\left(\frac{1}{\sigma(y)}\right)},\quad\text{and}\;\, \mathrm{Ind}_{F_{\sigma(y),y}}(x(y))\leq d.
	\end{align}
	As the set of $y\in X$ such that $F_{\sigma,y}$ is non-degenerate is dense, we can choose a sequence $\ens{y_k}_{k\in \N}\subset X$ such that $y_k\conv{k\rightarrow \infty}0$, such that $F_{\sigma_i,y_k}$ is non-degenerate for all $1\leq i\leq N_j$ for all $k\in \N$, and $\sigma(y_k)=\sigma_k^j\in I_j$ such that $\widetilde{F}_{\sigma_k,y_k}$ admits a critical point $x(y_k)=x_k\in X$ verifying \eqref{end}. As $I_j$ is compact, we can assume up to a subsequence that $\sigma_k^j\conv{k\rightarrow \infty}\sigma_{\infty}^j\in I_j$, and as 
	\begin{align}\label{boundend}
		\norm{\widetilde{F}_{\sigma_k,y_k}-F_{\sigma_k}}_{C^2(X)}\conv{k\rightarrow \infty}0,
	\end{align}
	we deduce that up to a subsequence,  by the \textbf{Energy bound} $(2)$ and \eqref{boundend}, we have (notice that $\D F_{\sigma_k,y_k}(x_k)=0$)
	\begin{align*}
		\norm{\D F_{\sigma}(x_k)}&\leq \norm{\D F_{\sigma}(x_k)-\D F_{\sigma_k}(x_k)}+\norm{\D F_{\sigma_k}(x_k)}\\
		&=\norm{\D F_{\sigma}(x_k)-\D F_{\sigma_k}(x_k)}+\norm{\D F_{\sigma_k}(x_k)-\D F_{\sigma_k,y_k}(x_k)}\conv{k\rightarrow \infty}0.
	\end{align*}
	Therefore, up to an additional subsequence and by the Palais-Smale condition, we have the strong convergence
	\begin{align*}
		x_k\conv{k\rightarrow \infty}x_{\infty}^j\in K(F_{\sigma_{\infty}^j}).
	\end{align*}
	Finally, by the strong convergence of the second derivative, we have
	\begin{align}\label{1sided}
		\mathrm{Ind}_{F_{\sigma_{\infty}^j}}(x_{\infty}^j)\leq \liminf_{k\rightarrow \infty}\mathrm{Ind}_{F_{\sigma_{k},y_k}}(x_k)
	\end{align} 
	and (notice that by non-degeneracy of $x_k$ for $F_{\sigma_k,y_k}$ that $\mathrm{Null}_{F_{\sigma_k,y_k}}(x_k)=0$)
	\begin{align}\label{2sided}
		\mathrm{Ind}_{F_{\sigma_{\infty}^j}}(x^j_{\infty})+\mathrm{Null}_{F_{\sigma_{\infty}^j}}(x^j_{\infty})\geq \limsup_{k\rightarrow \infty}\left(\mathrm{Ind}_{F_{\sigma_{k},y_k}}(x_k)+\mathrm{Null}_{F_{\sigma_k,y_k}}(x_k)\right)
	\end{align}
	so $x^j_{\infty}$ verifies \eqref{end} for $y=0$ and $\sigma(y)=\sigma_{\infty}^j$. Furthermore, if $\mathscr{A}$ is replaced by a dual family, then the one-sided estimate from below of the index is given by \eqref{2sided} while two sided estimates are given for co-dual, homological or cohomological families by \eqref{1sided} and \eqref{2sided}.
	
	 This concludes the proof of the theorem, as the sequences $\ens{\sigma_{\infty}^j}_{j\in \N}\subset(0,\infty)$ and $\ens{x_{\infty}^j}_{j\in \N}\subset X$ satisfy the conditions of the theorem.
\end{proof}

\nocite{}
\bibliographystyle{plain}
\bibliography{biblio}

\end{document}